\documentclass{amsart}
\usepackage{amssymb}
\usepackage{amsmath}
\usepackage{amsopn}
\usepackage{graphicx}
\usepackage{mathrsfs}
\usepackage{amsmath}
\usepackage{amsfonts}
\usepackage{amsthm}
\usepackage{mwe}
\usepackage{tikz}
\usepackage{pstool}
\usepackage{psfrag}
\usepackage{epstopdf}
\usepackage{pstricks}
\usepackage{tikz}
\usepackage{tikz-cd}
\usetikzlibrary{matrix,arrows,decorations.pathmorphing}
\usepackage{textcomp}     

\usetikzlibrary{matrix,arrows}
\usepackage{todonotes}
\usepackage{marginnote}
\usepackage{lettrine}
\usepackage{graphicx}
\usepackage{booktabs}
\usepackage[all]{xy}
\usepackage{enumitem}
\usepackage{latexsym}
\usepackage{eepic}
\usepackage{epsfig}
\usepackage{graphicx}
\usepackage{pst-node}
\usepackage{lipsum}
\usepackage{hyperref}
\hypersetup{
  colorlinks   = true, 
  urlcolor     = black, 
  linkcolor    = blue, 
  citecolor   = red 
}

\usepackage{amscd}
\usepackage{mathrsfs}
\usepackage[english]{babel}

\numberwithin{equation}{subsection}

\newtheoremstyle{plain2}    
   {}            
   {}            
   {\itshape}    
   {}            
   {\bfseries}   
   {.}           
   {5pt plus 1pt minus 1pt}  
   {{\thmnumber{#1} \thmname{#2}{\thmnote{ (#3)}}}}          
\theoremstyle{plain2}

\newtheorem{thm}[equation]{Theorem}
\newtheorem{theorem}[equation]{Theorem}
\newtheorem{cor}[equation]{Corollary}
\newtheorem{lemma}[equation]{Lemma}
\newtheorem{prop}[equation]{Proposition}

\newtheoremstyle{definition2}    
   {}
   {}
   {\normalfont}
   {}
   {\bfseries}
   {.}
   {5pt plus 1pt minus 1pt}
   {{\thmnumber{#1} \thmname{#2}{\thmnote{#3}}}}
\theoremstyle{definition2}

\newtheorem{rem}[equation]{Remark}

\newtheorem{example}[equation]{Example}

\newtheorem*{theorem*}{Theorem}

\newtheoremstyle{stepstyle}
   {}     {}
   {\normalfont}
   {\parindent}
   {\itshape}
   {}
   {5pt plus 1pt minus 1pt}
   {{\thmname{#1} \thmnumber{#2}:{\thmnote{#3}}}}
\theoremstyle{stepstyle}

\newtheoremstyle{point}
   {}     {}
   {\normalfont}
   {}
   {\bfseries}
   {}
   {5pt plus 1pt minus 1pt}
   {{\thmname{#1}(\thmnumber{#2})\thmnote{ #3.}}}
\theoremstyle{point}
\newtheorem{point}[equation]{}

\newtheoremstyle{subpoint}
   {}     {}            
   {\normalfont}
   {}                   
   {\normalfont}
   {}
   {5pt plus 1pt minus 1pt}
   {{\thmname{#1}{\bf (\thmnumber{#2})}\thmnote{ #3.}}}
\theoremstyle{subpoint}
\newtheorem{subpoint}[equation]{}

\newcommand{\spa}[1]{\begin{subpoint}#1\end{subpoint}}           

\newcommand{\D}{{\mathcal D}}

\newcommand{\N}{\ensuremath{\mathbb{N}}}
\newcommand{\Z}{\ensuremath{\mathbb{Z}}}
\newcommand{\Q}{\ensuremath{\mathbb{Q}}}

\newcommand{\R}{\ensuremath{\mathbb{R}}}
\newcommand{\C}{\ensuremath{\mathbb{C}}}
\newcommand{\CP}{\ensuremath{\mathbb{CP}}}

\newcommand{\cX}{\ensuremath{\mathscr{X}}}

\newcommand{\mm}{\ensuremath{\mathfrak{m}}}
\newcommand{\fp}{\ensuremath{\mathfrak{p}}}

\newcommand{\caM}{\ensuremath{\mathcal{M}}}
\newcommand{\caC}{\ensuremath{\mathcal{C}}}
\newcommand{\caO}{\ensuremath{\mathcal{O}}}
\newcommand{\caI}{\ensuremath{\mathcal{I}}}

\newcommand{\cC}{\ensuremath{\mathscr{C}}}

\newcommand{\cY}{\ensuremath{\mathscr{Y}}}
\newcommand{\cZ}{\ensuremath{\mathscr{Z}}}
\newcommand{\cW}{\ensuremath{\mathscr{W}}}

\renewcommand{\R}{\ensuremath{\mathbb{R}}}
\renewcommand{\C}{\ensuremath{\mathbb{C}}}

\renewcommand{\cY}{\ensuremath{\mathscr{Y}}}

\newcommand{\Hilb}{\ensuremath{\mathrm{Hilb}}}
\newcommand{\Spec}{\ensuremath{\mathrm{Spec}\,}}

\newcommand{\redu}{\mathrm{red}}
\newcommand{\Sym}{\mathrm{Sym}}
\newcommand{\pr}{\mathrm{pr}}

\newcommand{\Hom}{\mathrm{Hom}}

\newcommand{\an}{\mathrm{an}}

\newcommand{\divisor}{\mathrm{div}}
\newcommand{\weight}{\mathrm{wt}}

\newcommand{\rank}{\mathrm{rank}}

\newcommand{\Sk}{\mathrm{Sk}}
\newcommand{\gp}{\mathrm{gp}}

\title{The Essential Skeleton of a Product of Degenerations}
\author{Morgan V Brown}
\address{Department of Mathematics, University of Miami, Coral Gables, FL 33146 USA}
\email{mvbrown@math.miami.edu}
\author{Enrica Mazzon}
\address{Department of Mathematics, Imperial College London, 180 Queen's Gate, London SW72AZ, UK}
\email{e.mazzon15@imperial.ac.uk}
\date{}

\begin{document}
\maketitle

\begin{abstract}

We study the problem of how the dual complex of the special fiber of an snc degeneration $\cX_R$ changes under products. We view the dual complex as a skeleton inside the Berkovich space associated to $X_K$. Using the Kato fan, we define a skeleton $\Sk(\cX_R)$ when the model $\cX_R$ is log-regular. We show that if $\cX_R$ and $\cY_R$ are log-regular, and at least one is semistable, then $\Sk(\cX_R\times_R \cY_R) \simeq \Sk(\cX_R)\times \Sk(\cY_R)$. 

The essential skeleton $\Sk(X_K)$, defined by Musta\c{t}\u{a} and Nicaise, is a birational invariant of $X_K$ and is independent of the choice of $R$-model. We extend their definition to pairs, and show that if both $X_K$ and $Y_K$ admit semistable models,  $\Sk(X_K\times_K Y_K) \simeq \Sk(X_K)\times \Sk(Y_K)$. 

As an application, we compute the homeomorphism type of the dual complex of some degenerations of hyper-K{\"a}hler varieties. We consider both the case of the Hilbert scheme of a semistable degeneration of K3 surfaces, and the generalized Kummer construction applied to a semistable degeneration of abelian surfaces. In both cases we find that the dual complex of the $2n$-dimensional degeneration is homeomorphic to either a point, $n$-simplex, or $\mathbb{C}\mathbb{P}^n$, depending on the type of the degeneration.

\end{abstract}

\section{Introduction}

Let $R$ be a discrete valuation ring with quotient field $K$ and residue field $k$, and let $X$ be a smooth proper variety over $K$. While there may be no way to extend $X$ to a smooth proper variety over $R$, in $\text{char}(k)=0$ resolution of singularities guarantees that we can always produce an $R$ model $\cX$ where the special fiber $\cX_k$ is a strict normal crossings (snc) divisor. Given such a model, we associate the dual complex $\D(\cX_k)$, which is the dual intersection complex of the components of the special fiber. 

The dual complex of the special fiber of a such degeneration reflects the geometry of the generic fiber. If the generic fiber is rationally connected, then the dual complex of the special fiber is contractible \cite{deFernexKollarXu2012}. For Calabi-Yau varieties, degenerations are classified by the action of monodromy on the cohomology. The principle is that the degenerations with maximally unipotent actions have the most rich combinatorial structure in the dual complex. In this case, the dual complex is always a $\mathbb{Q}$-homology sphere, and Koll{\'a}r and Xu \cite{KollarXu} show that it is a sphere if $n\leqslant 3$ or $n \leqslant 4$ and the special fiber of a minimal dlt model is snc. 

The goal of this work is to understand the dual complex of a model for the product of two smooth proper varieties over $K$. We consider this problem from two perspectives.

\subsection{Skeletons of Berkovich spaces.} The first is via the theory of Berkovich spaces. In this setting we assume that $K$ is complete with respect to the valuation induced by $R$, which gives rise to a non-archimedean norm on $K$. In \cite{Berkovich1990}, Berkovich develops a theory of analytic geometry over $K$. He associates a $K$-analytic space to $X$; each point corresponds to a real valuation on the residue field of a point of $X$, extending the discrete valuation on $K$. This space, denoted by $X^\an$, is called the Berkovich space associated to $X$.

From any snc model $\cX$ of $X$ one can construct a subspace of $X^\an$, called the Berkovich skeleton of $\cX$ and denoted by $\Sk(\cX)$: it is homeomorphic to the dual intersection complex of the divisor $\cX_k$ \cite{MustataNicaise}.
The Berkovich skeletons turn out to be relevant in the study of the topology of $X^\an$.  They shape the Berkovich space, as $X^\an$ is homeomorphic to the inverse limit $\underleftarrow{\lim}\Sk(\cX)$ where $\cX$ runs through all snc models of $X$. Also, the homotopy type of $X^\an$ is determined by any snc model $\cX$: indeed, Berkovich and Thuillier prove that $\Sk(\cX)$ is a strong deformation retract of $X^\an$ \cite{Berkovich1990, Thuillier2007}.

\subsection{The dual complex of a dlt model.} The other approach to the study of the dual complexes comes from birational geometry. In this setting, we consider a pair $(\cX, \Delta_\cX)$ over the germ of a curve. In the log general type case, running the minimal model program distinguishes a canonical model for the degeneration $(\cX, \Delta_\cX)$ \cite{Alexeev, KollarShepherd-Barron, HaconMcKernanXu}, at the cost of possibly introducing worse singularities. If we are willing to tolerate some ambiguity in our choice of model, we can choose instead to produce a dlt minimal model. One advantage of dlt models is that they are expected to exist for all pairs admitting a log pluricanonical form. The singularities are mild enough that it is possible to define the dual complex as the dual intersection complex of divisors of coefficient $1$, denoted $\D^{=1}(\cX,\Delta_\cX)$. In \cite{deFernexKollarXu2012}, de Fernex, Koll{\'a}r, and Xu investigate how the dual complex is affected by the operations of the minimal model program. They show, under mild hypotheses, that every step of the MMP induces a homotopy equivalence between dual complexes. Moreover, $\D^{=1}(\cX,\Delta_\cX)$ is a PL invariant under log crepant birational maps.

\subsection{The essential skeleton.} Recently there has been much interest in a synthesis of the two approaches. Kontsevich and Soibelmann \cite{KontsevichSoibelman} define a version of the skeleton of a variety with trivial canonical bundle, which detects the locus of simple poles along the special fiber of the distinguished canonical form. Musta{\c{t}}{\u{a}} and Nicaise \cite{MustataNicaise} extend their definition to any variety with non-negative Kodaira dimension. The key technical tool is the definition, for a rational pluricanonical form, of a weight function on the Berkovich space. The essential skeleton $\Sk(X)$ is the union over all regular pluricanonical forms of the minimality locus of the associated weight functions.

Thus the essential skeleton has the advantage of being intrinsic to the variety $X$, with no dependence on a choice of model. As the weight function is closely related to the log discrepancy from birational geometry, it is natural to expect that the essential skeleton in some way encodes some of the minimal model theory of $X$. Nicaise and Xu \cite{NicaiseXu} show, when $X$ is a smooth projective variety with $K_X$ semiample, and $\cX$ is a good dlt minimal model, that the dual complex of $\cX_k$ can be identified with the essential skeleton of $X$. While it is in general a difficult problem to produce good dlt minimal models, Koll{\'a}r, Nicaise and Xu \cite{KollarNicaiseXu} show that for any smooth projective $X$ with $K_X$ semiample, $X$ extends to a good dlt minimal model over a finite extension of the valuation ring.

\subsection{Skeletons for log-regular models.} To produce nice models of the product, we work in the context of the logarithmic geometry. To any log-regular scheme $\cX^+$, in \cite{Kato1994a} Kato attaches a combinatorial structure $F_\cX$ called a \emph{fan}: if we denote by $D_\cX$ the locus where the log structure is non-trivial, then the fan $F_\cX$ consists of the set of the generic points of intersections of irreducible components of $D_\cX$, equipped with a sheaf of monoids. We define a logarithmic version of the Berkovich skeleton for a log-regular model $\cX^+$ of $X$ over $R$: it gives rise to a polyhedral complex in $X^\an$ whose faces correspond to the points of $F_\cX$.

Given two log-smooth log schemes $\cX^+$ and $\cY^+$ over $R$, their product $\cZ^+$ in the category of fine and saturated log schemes is naturally log-regular, hence $\cZ^+$ has an associated skeleton, and it is a model of the product $\cX_K \times_K \cY_K$ of the generic fibers. If one of the two underlying schemes $\cX$ or $\cY$ is semistable, which means it has reduced special fiber, then we show that the skeleton of the product $\cZ^+$ is the product of the skeletons, with the projection maps given by restricting the valuation to the corresponding function fields (Proposition \ref{prop semistability and skeletons}).

\subsection{Skeletons for pairs.} Working in the logarithmic setting, we may also allow a non-trivial log structure over the generic fiber. Geometrically this corresponds to adding horizontal divisors to the special fiber and yields the addition of some unbounded faces to the skeleton. In \cite{GublerRabinoffWerner} Gubler, Rabinoff, and Werner construct a skeleton for strictly semistable snc models with suitable horizontal divisors. Both constructions recover the Berkovich skeleton when there is no horizontal component and the special fiber is snc.

Pairs arise frequently in the minimal model program. Taking advantage of a construction that admits horizontal components, the definition of the essential skeleton extends to the case of a pair $(X,\Delta_X)$ over $K$, where $\Delta_X$ has $\mathbb{Q}$-coefficients in $[0,1]$, the support of $\Delta_X$ is snc, and to pluricanonical forms of some positive index $r$ with divisor of poles no worse than $r\Delta_X$. We extend to pairs the result of Musta{\c{t}}{\u{a}} and Nicaise \cite{MustataNicaise} on the birational invariance of the essential skeleton (Proposition \ref{prop birational invariance essential skeleton for pairs}), as well as Nicaise and Xu's result \cite{NicaiseXu} that the essential skeleton is homeomorphic to the dual complex of a good minimal dlt model (Proposition \ref{prop dual complex minimal good dlt = essential skeleton any resolution}). It follows from these results that we can define the notion of essential skeleton for a dlt pair.

\subsection{Main Result.} Our main result establishes the behavior of essential skeletons under products.
\begin{theorem}\label{main thm intro}Assume that the residue field $k$ is algebraically closed. Let $(X,\Delta_X)$ and $(Y,\Delta_Y)$ be pairs that induce log-regular structures. Suppose that both pair have non-negative Kodaira-Iitaka dimension and admit semistable log-regular models $\cX^+$ and $\cY^+$ over $S^+$. Then, the PL homeomorphism of skeletons induces a PL homeomorphism of essential skeletons $$\Sk(Z,\Delta_Z) \xrightarrow{\sim} \Sk(X,\Delta_X) \times \Sk(Y,\Delta_Y)$$ where $Z$ and $\Delta_Z$ are the respective products. 
\end{theorem}
Semistability is a key assumption; without it the projection map might fail to be injective, see Example \ref{quartic}. As expected, we get a corresponding result for dual complexes of semistable good dlt minimal models (Theorem \ref{thm product of dual complex of good min dlt models}). Unfortunately semistability is not well behaved under birational transformations so it seems possible that a degeneration admits a semistable good dlt minimal model but no semistable log-regular model.

\subsection{Application to degenerations of hyper-K{\"a}hler varieties.}

As an application of Theorem \ref{main thm intro}, we study certain degenerations of hyper-K{\"a}hler varieties. One way to produce hyper-K{\"a}hler varieties is by taking the Hilbert scheme of points on a K3 surface. Another is to extend the Kummer construction to higher dimensional abelian varieties. Aside from two other examples found by O'Grady in dimensions 6 \cite{OGrady} and 10 \cite{OGradya} there are no other known examples, up to deformation equivalence.

Just as for Calabi-Yau varieties, degenerations of hyper-K{\"a}hler varieties can be understood in terms of the monodromy operator on cohomology, with classification into types I, II, III. Type I is the case where the dual complex is just a single point, but types II and III have more interesting structure. Koll{\'a}r, Laza, Sacc{\`a} and Voisin \cite{KollarLazaSaccaEtAl2017} show that in the type II case the dual complex has the rational homology type of a point, and in the type III case of a complex projective space. Gulbrandsen, Halle, and Hulek  \cite{GulbrandsenHalleHulek, GulbrandsenHalleHulek2016, GulbrandsenHalleHulekEtAl} use GIT to construct a model of the degeneration of $n$-th order Hilbert schemes arising from some type II degenerations of K3 surfaces, and show that the dual complex is an $n$-simplex. There are considerations from mirror symmetry that suggest that for a type III degeneration the dual complex should be homeomorphic to $\mathbb{C}\mathbb{P}^n$ \cite{Hwang, KontsevichSoibelman}.

\begin{theorem}
Assume that the residue field $k$ is algebraically closed. Let $S$ be a $\text{K}3$ surface over $K$. If $S$ admits a semistable log-regular model or a semistable good dlt minimal model, then the essential skeleton of the Hilbert scheme of $n$ points on $S$ is isomorphic to the $n$-th symmetric product of the essential skeleton of $S$ $$\Sk(\Hilb^n(S)) \xrightarrow{\sim} \Sym^n(\Sk(S)).$$ 
\end{theorem}

Computing these complexes gives a single point in the type I case, an $n$-simplex in the type II case, and $\mathbb{C}\mathbb{P}^n$ in the type III case. The same types arise in the Kummer case.

\begin{theorem}
Assume the residue field $k$ is algebraically closed. Let $A$ be an abelian surface over $K$. Suppose that $A$ admits a semistable log-regular model or a semistable good dlt minimal model. If the essential skeleton of $A$ is homeomorphic to a point, the circle $S^1$ or the torus $S^1 \times S^1$, then the essential skeleton of the $n$-th generalised Kummer variety $\text{K}_n(A)$ is isomorphic to a point, the standard $n$-simplex or $\CP^n$ respectively. 
\end{theorem}

Analysis of the weight function gives a powerful yet accessible approach to controlling the skeletons of these varieties. In both cases we use Theorem \ref{main thm intro} to establish that the skeleton of the hyper-K{\"a}hler variety is a finite quotient of the $n$-fold product of the skeleton of the original surface under the action of a symmetric group. In the case of Hilbert schemes we can get a complete description of the action using functoriality of the projection maps, but in the Kummer case we additionally need to understand the restriction of the multiplication map to the essential skeleton of an abelian surface (\cite{Berkovich1990, HalvardHalleNicaise2017, Temkina}).

To our knowledge these are the first examples of type III degenerations of hyper-K{\"a}hler varieties where the PL homeomorphism type of the dual complex is known. 

\subsection{Structure of the paper.}
In section \ref{sect kato fan} we recall the Kato fan, and develop necessary background for the rest of the paper. Section \ref{sect skeleton log} contains the definition and properties of the skeleton of a log-regular scheme. These first two sections require no hypothesis on the characteristic of $K$. 

For the remainder of the paper, we assume that $K$ has equicharacteristic $0$. In section \ref{sec weight} we introduce the weight function and essential skeleton from \cite{MustataNicaise}, extending their definition to the case of a pair. In section \ref{sec birational} we give connections to birational geometry, along with proofs of the main theorems. We discuss an application to degenerations of hyper-K{\"a}hler varieties in section \ref{sec hkahler}.

\subsection{Acknowledgements.}
The authors would like to thank Drew Armstrong, Lorenzo Fantini, Phillip Griffiths, Lars Halvard Halle, Sean Keel, Janos Koll{\'a}r, Mirko Mauri, Johannes Nicaise, Giulia Sacc{\`a}, Chenyang Xu, and David Zureick-Brown for helpful conversations, comments, and suggestions. The authors are especially grateful to Johannes Nicaise, Mazzon's advisor, both for suggesting this project and for putting the authors in contact. Mazzon is grateful to the University of Miami for hosting her visit during the completion of this project. Brown was supported by the Simons Foundation Collaboration Grant 524003.
Mazzon was partially supported by the ERC Starting Grant MOTZETA (project 306610) of the European Research Council (PI: Johannes Nicaise), and by the Engineering and Physical Sciences Research Council [EP/L015234/1]. The EPSRC Centre for Doctoral Training in Geometry and Number Theory (The London School of Geometry and Number Theory), University College London.

\subsection{Notation.}
\spa{Let $R$ be a complete discrete valuation ring with maximal ideal $\mm$, residue field $k=R/\mm$ and quotient field $K$. We assume that the valuation $v_K$ is normalized, namely $v_K(\pi)=1$ for any uniformizer $\pi$ of $R$. We define by $|\cdot|_K= \exp(- v_K(\cdot))$ the absolute value on $K$ corresponding to $v_K$; this turns $K$ into a non-archimedean complete valued field.}

\spa{ We write $S=\Spec R$ and we denote by $s$ the closed point of $S$. Let $\cX$ be an $R$-scheme of finite type. We will denote by $\cX_k$ the special fiber of $\cX$ and by $\cX_K$ the generic fiber. 
}

\spa{Let $X$ be a proper $K$-scheme. A model for $X$ over $R$ is a flat separated $R$-scheme $\cX$ of finite type endowed with an isomorphism of $K$-schemes $\cX_K \rightarrow X$. If $X$ is smooth over $K$,
we say that $\cX$ is an snc model for $X$ if it is proper, regular over $R$, and the
special fiber $\cX_k$ is a strict normal crossings divisor on $\cX$. In equicharacteristic $0$, such a model always exists, by Hironaka's resolution of singularities.

We say that a model $\cX$ over $R$ (not necessarily regular) is semistable if the special fiber $\cX_k$ is reduced.}

\spa{ \label{sss-log} All log schemes in this paper are  fine and saturated (\emph{fs}) log schemes and defined with respect to the Zariski topology. We denote a log scheme by $\cX^+=(\cX,\caM_{\cX})$, where $\caM_{\cX}$ is the structural sheaf of monoids. We denote by $$\mathcal{C}_{\cX}=\mathcal{M}_{\cX}/\mathcal{O}_{\cX}^{\times}$$the characteristic sheaf of $\cX^+$. The sheaf $\mathcal{C}_{\cX}$ is a Zariski sheaf on $\cX^+$, supported on $\cX_k$; if $\cX^+$ is log-regular, then $\caC_{\cX}$ is a constructible sheaf. For every point $x$ of $\cX_k$, we denote by $\mathcal{I}_{\cX,x}$ the ideal in $\mathcal{O}_{\cX,x}$ generated by $$\mathcal{M}_{\cX,x}\setminus \mathcal{O}_{\cX,x}^{\times}.$$

If a log scheme $\cX^+$ has divisorial log structure induced from a divisor $D$, we denote it by $\cX^+=(\cX,D)$. 

We denote by $S^+$ the scheme $S$ endowed with the standard log structure (the divisorial log structure induced by $s$), namely $S^+=(S,s)$. If an $R$-scheme $\cX$ is given, we will always denote by $\cX^+$ the log scheme over $S^+$ that we obtain by endowing $\cX$ with the divisorial log structure associated with
$\cX_k$.}

\spa{A log scheme is log-regular at a point $x$ if the following two conditions are satisfied: $\mathcal{O}_{\cX,x}/\mathcal{I}_{\cX,x}$ is a regular local ring, and $\dim \mathcal{O}_{\cX,x} = \dim \mathcal{O}_{\cX,x}/\mathcal{I}_{\cX,x}+\mathrm{rank}\,\mathcal{C}^{\gp}_{\cX,x}$. For example, a toric variety with its toric logarithmic structure is log-regular. More generally, working over perfect fields, log-regular varieties correspond to toroidal embeddings (without self-intersections).

If $\cX^+$ is a log-regular log scheme over $S^+$, then the locus where the log structure is non-trivial is a divisor that we will denote by $D_{\cX}$. Thus, the log structure on $\cX^+$ is the divisorial log structure induced by $D_{\cX}$, by \cite{Kato1994a}, Theorem 11.6.

For a more extended dissertation of the theory of log schemes we refer to \cite{Kato1989,GabberRamero, Kato1994}.
}

\spa{Let $(X,\Delta_X)$ be an snc pair where $X$ is a proper $K$-scheme and $\Delta_X$ is an effective $\Q$-divisor such that $\Delta_X= \sum a_i \Delta_{X,i}$ with $0 \leqslant a_i \leqslant 1$. A log-regular log scheme $\cX^+$ over $S^+$ is a model for $(X,\lceil \Delta_X \rceil)$ over $S^+$ if $\cX$ is a model of $X$ over $R$, the closure of any component of $\Delta_X$ in $\cX$ has non-empty intersection with $\cX_k$, and $D_\cX = \overline{\lceil \Delta_X \rceil} + \cX_{k,\redu}$. }


\spa{We denote by $(\cdot)^{\an}$ the analytification functor from the category of $K$-schemes of finite type to Berkovich's category of $K$-analytic spaces. For every $K$-scheme of finite type $X$, as a set, $X^{\an}$ consists of the pairs $x=(\xi_x,|\cdot|_x)$ where $\xi_x$ is a point of $X$ and $|\cdot|_x$ is an absolute value on the residue field $\kappa(\xi_x)$ of $X$ at $\xi_x$ extending the absolute value $|\cdot|_K$ on $K$. We endow $X^{\an}$ with the Berkovich topology, i.e. the weakest one such that \begin{itemize}
\item[(i)] the forgetful map $\phi: X^{\an} \rightarrow X$, defined as $(\xi_x,|\cdot|_x) \mapsto \xi_x$, is continuous,
\item[(ii)] for any Zariski open subset $U$ of $X$ and any regular function $f$ on $U$ the map $|f|:\phi^{-1}(U) \rightarrow \R$ defined by $|f|(\xi_x,|\cdot|_x)=|f(\xi_x)|$ is continuous.  
\end{itemize} The set $\text{Bir}(X)$ of birational points of $X^\an$ is defined as the inverse image under $\phi$ of the generic point of $X$. By definition, it is a birational invariant of $X$.}

\section{The Kato fan of a log-regular log scheme} \label{sect kato fan}

\subsection{Definition of Kato fans.}
\spa{ According to \cite{Kato1994a}, Definition 9.1, a monoidal space $(T, \caM_T)$ is a topological space $T$ endowed with a sharp sheaf of monoids $\caM_T$, where \textit{sharp} means that $\caM_{T,t}^\times = \{1\}$ for every $t \in T$. We often simply denote the monoidal space by $T$.

A morphism of monoidal spaces is a pair $(f,\varphi):(T,\caM_T) \rightarrow (T',\caM_{T'})$  such that $f:T \rightarrow T'$ is a continuous function of topological spaces and $\varphi: f^{-1}(\caM_T) \rightarrow \caM_{T'}$ is a sheaf homomorphism such that $\varphi_{t}^{-1}(\{1\})=\{1\}$ for every $t \in T$.}

\begin{example}
If $\cX^+$ is a log scheme then the Zariski topological space $\cX$ is equipped with a sheaf of sharp monoids $\caC_{\cX}$, namely the characteristic sheaf of $\cX^+$. Thus $(\cX,\caC_{\cX})$ is a monoidal space. Moreover, morphisms of log schemes
induce morphisms of characteristic sheaves, hence morphism of monoidal spaces. We therefore obtain a functor from the category of log schemes to the category of monoidal spaces.
\end{example}

\begin{example}
Given a monoid $P$, we may associate to it a monoidal space called the spectrum of $P$. As a set, $\Spec P$ is the set of all prime ideals of $P$. The topology is characterized by the basis open sets $D(f)= \{ \fp \in \Spec P | f \notin \fp \}$ for any $f \in P$. The monoidal sheaf is defined by $$\caM_{\Spec P} (D(f))= S^{-1}P / (S^{-1}P)^{\times}$$ where $S=\{f^n | n \geqslant 0\}$.
\end{example}

\spa{A monoidal space isomorphic to the monoidal space $\Spec P$ for some monoid $P$ is called an affine Kato fan. A monoidal space is called a Kato fan if it has an open covering consisting of affine Kato fans. In particular, we call a Kato fan integral, saturated, of finite type or $fs$ if it admits a cover by the spectra of monoids with the respective properties.}

\spa{A morphism of fs Kato fans $F' \rightarrow F$ is called a \emph{subdivision} if it has finite fibers and the morphism $$\Hom(\Spec \N, F') \rightarrow \Hom (\Spec \N, F)$$ is a bijection. By allowing subdivisions, a Kato fan might take the following shape.
\begin{prop} \label{kato subdivision}(\cite{Kato1994a}, Proposition 9.8)
Let $F$ be a $fs$ Kato fan. Then there is a subdivision $F' \rightarrow F$ such that $F'$ has an open cover $\{U'_i\}$ by Kato cones with $U'_i \simeq \Spec \N^{r_i}$. 
\end{prop}
The strategy of the proof of Proposition \ref{kato subdivision} goes back to \cite{KempfKnudsenMumfordEtAl1973} and relies on a sequence of particular subdivisions of the Kato fan, the so-called star and barycentric subdivisions (\cite{AbramovichChenMarcusEtAl2015}, Example 4.10).}

\subsection{Kato fans associated to log-regular log schemes.}
\begin{thm}\label{kato fan associated} (\cite{Kato1994a}, Proposition 10.2)
Let $\cX^+$ be a log-regular log scheme. Then there is an initial strict morphism $(\cX, \caC_{\cX}) \rightarrow F$ to a Kato fan in the category of monoidal spaces. Explicitly, there exist a Kato fan $F$ and a morphism $\varrho: (\cX, \caC_{\cX}) \rightarrow F$ such that $\varrho^{-1}(\caM_{F}) \simeq \caC_{\cX}$ and any other morphism to a Kato fan factors through $\varrho$.
\end{thm}
The Kato fan $F$ in Theorem \ref{kato fan associated} is called the Kato fan associated to $\cX^+$; concretely, it is the topological subspace of
$\cX$ consisting of the points $x$ such that the maximal ideal $\mathfrak{m}_x$ of
$\mathcal{O}_{\cX,x}$ is equal to $\mathcal{I}_{\cX,x}$, and
$\mathcal{M}_F$ is the inverse image of $\mathcal{C}_{\cX}$ on
$F$, henceforth we write $\caC_F$ for $\caM_F$.

\begin{example} \label{example kato fan regular model}
Assume that $\cX$ is regular, of finite type over $S$ and $\cX_k$ is a divisor with strict
normal crossings. Then $\cX^+$ is log-regular and $F$ is the set
of generic points of intersections of irreducible components of
$\cX_k$. For each point $x$ of $F$, the stalk of $\mathcal{C}_F$
is isomorphic to $(\N^r,+)$, with $r$ the number of irreducible
components of $\cX_k$ that pass through $x$.
\end{example}

This example admits the following partial generalisation.

\begin{lemma} \label{lemma points Kato fan}
Let $\cX^+$ be a log-regular log scheme. Then the fan $F$ consists of the generic points of intersections of irreducible components of $D_{\cX}$.
\end{lemma}
\begin{proof}
First, we show that every such generic point is a point of $F$. Let $E_1,\ldots,E_r$ be irreducible components of $D_{\cX}$ and let $x$ be a generic point of the intersection $E_1\cap\ldots\cap E_r$. We set $d=\dim \mathcal{O}_{\cX,x}$. Since $\cX^+$ is log-regular, we know that $\mathcal{O}_{\cX,x}/\mathcal{I}_{\cX,x}$ is regular and that \begin{equation} \label{equ log regularity}
d=\dim \mathcal{O}_{\cX,x}/\mathcal{I}_{\cX,x}+\mathrm{rank}\,\mathcal{C}^{\gp}_{\cX,x}.\end{equation} We denote by $V(\caI_{\cX,x})$ the vanishing locus of the ideal $\caI_{\cX,x}$ in $\cX$.
We want to prove that $\caI_{\cX,x}= \mathfrak{m}_x$. We assume the contrary, hence that $\caI_{\cX,x} \subsetneq \mathfrak{m}_x$. This assumption implies that there exists $j$ such that $V(\caI_{\cX,x}) \nsubseteq E_j$: indeed, if the vanishing locus is contained in each irreducible component $E_i$, i.e. $$V(\caI_{\cX,x}) \subseteq E_1 \cap \ldots \cap E_r \subseteq \overline{\{x\}},$$ then $\caI_{\cX,x} \supseteq \mathfrak{m}_x$. From the assumption of log-regularity it follows that the vanishing locus $V(\caI_{\cX,x})$ is a regular subscheme, and moreover that $\cX^+$ is Cohen-Macaulay by \cite{Kato1994a}, Theorem 4.1. Thus, there exists a regular sequence $(f_1,\ldots,f_l)$ in $\caI_{\cX,x}$, where $l$ is the codimension of $V(\caI_{\cX,x})$, i.e. $$ \dim \mathcal{O}_{\cX,x}/\mathcal{I}_{\cX,x} = d-l.$$ Moreover by the equality (\ref{equ log regularity}), $\mathrm{rank}\,\mathcal{C}^{\gp}_{\cX,x} = l$.
 
We claim that the residue classes of these elements $f_i$ in $\mathcal{C}_{\cX,x}^{\gp}$ are linearly independent. Assume the contrary. Then, up to renumbering the $f_i$, there exist an  integer $e$ with $1< e< l$, non-negative integers $a_1,\ldots,a_l$, not all zero, and a unit $u$ in $\mathcal{O}_{\cX,x}$ such that $$f_1^{a_1}\cdot \ldots \cdot f_{e-1}^{a_{e-1}}=u\cdot f_e^{a_e}\cdot \ldots \cdot f_l^{a_l}.$$ This contradicts the fact that $(f_1,\ldots,f_l)$ is a regular sequence in $\mathcal{I}_{\cX,x}$. Thus, the classes $\overline{f_1}, \ldots,\overline{f_l}$ are independent in $\caC_{\cX,x}^\gp$. As we also have the equality $\mathrm{rank}\,\mathcal{C}^{\gp}_{\cX,x}=l$, it follows that these classes generate $\caC_{\cX,x}^\gp \otimes_{\Z} \Q$.

Let $g_j$ be a non-zero element of the ideal $\caI_{\cX,x}$ that vanishes along $E_j$: it necessarily exists as otherwise $E_j$ is not a component of the divisor $D_{\cX}$. Then $g_j$ satisfies $$g_{j}^N = v \cdot f_1^{b_1} \cdot \ldots \cdot f_l^{b_l}$$ with $b_i \in \Z$, $v$ is a unit in $\caO_{\cX,x}$ and $N$ is a positive integer. As $g_j$ vanishes along the irreducible component $E_j$, at least one of the functions $f_1,\ldots,f_l$ has to vanish along $E_j$: assume that is $f_1$.

On the one hand, as $f_1$ is identically zero on $E_j$, the trace of $E_j$ on $V(\caI_{\cX,x})$ has at most codimension $l-1$ in $E_j$ at the point $x$. On the other hand, we assumed that $V(\caI_{\cX,x})$ is not contained in $E_j$ and it has codimension $l$ in $\caO_{\cX,x}$. Then, the trace of $E_j$ on $V(\caI_{\cX,x})$ has codimension $l$ in $E_j$ at $x$. This is a contradiction. We conclude that the ideal $\caI_{\cX,x}$ is equal to the maximal ideal $\mathfrak{m}_x$, therefore $x$ is a point of $F$.

It remains to prove the converse implication: every point $x$ of the fan $F$ is a generic point of an intersection of irreducible components of $D_{\cX}$. Let $x$ be a point of $F$: by construction of Kato fan $F$, the maximal ideal of $\mathcal{O}_{\cX,x}$ is equal to $\mathcal{I}_{\cX,x}$, thus it is generated by elements in $\mathcal{M}_{\cX,x}$. The zero locus of such an element is contained in $D_{\cX}$ by definition of the logarithmic structure on $\cX^+$. Therefore, the zero locus of a generator of $\mathfrak{m}_x$ in $\mathcal{M}_{\cX,x}$ is a union of irreducible components of the trace of $D_{\cX}$ on $\Spec \mathcal{O}_{\cX,x}$ and $x$ is a generic point of the intersection of all such irreducible components.
\end{proof}

\begin{rem}
By convention, the generic point of the empty intersection of irreducible components is the generic point of $\cX$. By definition, this point is also included in the Kato fan $F$. Thus, for example, the Kato fan associated to $S^+$ consists of two points: the generic point of $S$ that corresponds to the empty intersection, and the closed point $s$ corresponding to the unique irreducible component of the logarithmic divisorial structure.  
\end{rem}

Moreover, the example \ref{example kato fan regular model} also leads to the following characterization.

\begin{prop} \label{resolution log scheme kato fan} (\cite{GabberRamero}, Corollary 12.5.35)
Let $\cX^+$ be a log-regular log scheme over $S^+$ and $F$ its associated Kato fan. The following are equivalent: \begin{enumerate}
\item for every $x \in F$, $M_{F,x} \simeq \N^{r(x)}$,
\item the underlying scheme $\cX$ is regular.
\end{enumerate}
If this is the case, then the special fiber $\cX_k$ is a strict normal crossing divisor.
\end{prop}
\spa{The construction of the Kato fan of a log scheme defines a functor from the category of log-regular log schemes to the category of Kato fans. Indeed, given a morphism of log schemes $\cX^+ \rightarrow \cY^+$, we consider the embedding of the associated Kato fan $F_{\cX}$ in $\cX^+$ and the canonical morphism $\cY^+ \rightarrow F_{\cY}$: the composition $$F_{\cX} \hookrightarrow \cX^+ \rightarrow \cY^+ \rightarrow F_\cY$$ functorially induces a map between associated Kato fans. Moreover, this association preserves strict morphisms (\cite{Ulirsch}, Lemma 4.9).}
\subsection{Resolutions of log schemes via Kato fan subdivisions.} \label{section resolution via kato subd}
\begin{prop} (\cite{Kato1994a}, Proposition 9.9) \label{prop morphism induced by subdivision kato fan}
Let $\cX^+$ be a log-regular log scheme and let $F$ be its associated Kato fan. Let $F' \rightarrow F$ be a subdivision of fans. Then there exist a log scheme ${\cX'}^+$, a morphism of log schemes ${\cX'}^+ \rightarrow \cX^+$ and a commutative diagram
\begin{center}
$\begin{gathered}
\xymatrixcolsep{2pc} \xymatrix{
  (\cX', \caC_{\cX'}) \ar[r]^-{p} \ar[d]^{} & F' \ar[d]^{}\\
  (\cX, \caC_{\cX}) \ar[r]^-{\pi_{\cX}}   & F
}
\end{gathered}$\end{center}such that $p^{-1}(\caM_{F'}) \simeq \caC_{\cX'}$; they define a final object in the category of such diagrams and the refinement $F'\to F$ is induced by the morphism of log-regular log schemes $\cX'^+ \rightarrow \cX^+$.
\end{prop}

\spa{\label{rem resolution via subd}
It follows that given any subdivision $F' \rightarrow F$ of the Kato fan F associated with a log-regular log scheme $\cX^+$, we can construct a log scheme over $\cX^+$ with prescribed associated Kato fan $F'$. Combining this fact with Proposition \ref{kato subdivision} and Proposition \ref{resolution log scheme kato fan} yields to the construction of resolutions of log schemes in the following sense: for any log-regular log scheme over $S^+$ we can find a birational modification by a regular log scheme with strict normal crossings special fiber. Moreover, the morphism of log schemes ${\cX'}^+ \rightarrow \cX^+$ is obtained by a log blow-up  (\cite{Niziol2006}, Theorem 5.8).}

\subsection {Fibred products and associated Kato fans.}\label{sect fs product}
\spa{Given morphisms of \emph{fs} log schemes $f_1: \cX_1^+ \rightarrow \cY^+$ and $f_2: \cX_2^+ \rightarrow \cY^+$, their fibred product exists in the category of log schemes. It is obtained by endowing the usual fibred product of schemes
\begin{equation} \label{fibred product diagram}
\begin{gathered}
\xymatrixcolsep{3pc} \xymatrix{
  \cX_1\times_{\cY} \cX_2 \ar[r]^-{p_1} \ar[dr]^-{p_{\cY}} \ar[d]^-{p_2} & \cX_1 \ar[d]^{f_1}\\
  \cX_2 \ar[r]^{f_2}   & \cY
}
\end{gathered}
\end{equation}
with the log structure associated to $p_1^{-1}\caM_{\cX_1} \oplus_{p_{\cY}^{-1}\caM_{\cY}} p_2^{-1}\caM_{\cX_2}$. If $u_1:P \rightarrow Q_1$ and $u_2:P \rightarrow Q_2$ are charts for the morphisms $f_1$ and $f_2$ respectively, then the induced morphism $\cX_1\times_{\cY} \cX_2 \rightarrow \Spec \Z[Q_1 \oplus_P Q_2]$ is a chart for $\cX_1^+\times_{\cY^+} \cX_2^+$. }

\spa{In general, the fibred product is not $fs$, but the category of $fs$ log schemes also admits fibred products. Keeping the same notations, the following is a chart of the fibred product in the category of fine and saturated log schemes $$\cX_1^+\times_{\cY^+}^{\text{fs}} \cX_2^+ = (\cX_1^+\times_{\cY^+} \cX_2^+) \times_{\Z[Q_1 \oplus_P Q_2]} \Z[(Q_1 \oplus_P Q_2)^{\text{sat}}]$$(\cite{Bultot2015}, 3.6.16). We remark that the two fibred products above may not only have different log structures, but also the underlying schemes may differ. Nevertheless, this obviously does not occur when the monoid $Q_1 \oplus_P Q_2$ is saturated.}

\spa{Log-smoothness is preserved under fs base change and composition (\cite{GabberRamero}, Proposition 12.3.24). In particular, if $f_1: \cX_1^+ \rightarrow \cY^+$ is log-smooth and $\cX_2^+$ is log-regular, then $\cX_1^+\times^{\text{fs}}_{\cY^+} \cX_2^+$ is log-regular, by \cite{Kato1994a}, Theorem 8.2.

Consider log-smooth morphisms of fs log schemes $\cX_1^+ \rightarrow \cY^+$ and $\cX_2^+ \rightarrow \cY^+$. The sheaves of logarithmic differentials are related by the following isomorphism \begin{equation} \label{equ kahler diff product}
p_1^* \Omega^{\log}_{\cX_1^+/\cY^+} \oplus p_2^* \Omega^{\log}_{\cX_2^+/ \cY^+} \simeq \Omega^{\log}_{\cX_1^+\times^{\text{fs}}_{\cY^+} \cX_2^+ /\cY^+}\end{equation} by \cite{GabberRamero}, Proposition 12.3.13. Furthermore, by assumption of log-smoothness over $S^+$ the logarithmic differential sheaves are locally free of finite rank (\cite{Kato1994a}, Proposition 3.10) and we can consider their determinants; they are called log canonical bundles and denoted by $\omega^{\log}$. The following isomorphism is a direct consequence of (\ref{equ kahler diff product}) \begin{equation} \label{equ log can bundles}
p_1^* \omega^{\log}_{\cX_1^+/\cY^+} \otimes p_2^* \omega^{\log}_{\cX_2^+/ \cY^+} \simeq \omega^{\log}_{\cX_1^+\times^{\text{fs}}_{\cY^+} \cX_2^+ /\cY^+}.\end{equation} }

\spa{Similarly to the construction of fibred products of $fs$ log schemes, the category of $fs$ Kato fans admits fibred products: on affine Kato fans $F=\Spec P$ and $G=\Spec Q$ over $H=\Spec T$, $F\times_{H}G$ is the spectrum of the amalgamated sum $(P \oplus_T Q)^{\text{sat}}$ in the category of $fs$ monoids (\cite{Ulirsch2016}, Proposition 2.4) and on the underlying topological spaces, this coincides with the usual fibred product.}
We seek to compare the Kato fan associated to the fibred product of log-regular log schemes with the fibred product of associated Kato fans.
\begin{prop} \label{local isomo kato fan product}(\cite{Saito2004}, Lemma 2.8) Given $\mathscr{T}^+$ a log-regular log scheme, let $\cX^+$ and $\cY^+$ be log-smooth log schemes over $\mathscr{T}^+$. We denote by $\cZ^+$ the $fs$ fibred product $\cX^+\times^{\text{fs}}_{\mathscr{T}^+} \cY^+$. Then, the natural morphisms $F_{\cZ}\to F_{\cX}$ and $F_{\cZ}\to F_{\cY}$ induce a morphism of Kato fans \begin{equation} \label{equ local isomo kato fans}
F_{\cZ} \rightarrow F_{\cX} \times_{F_{\mathscr{T}}} F_{\cY}
\end{equation} that is locally an isomorphism.
\end{prop}

\spa{For any pair of points $(x,y)$ in $F_\cX \times_{F_{\mathscr{T}}} F_\cY$, we denote by $n(x,y)$ the number of preimages of $(x,y)$ in the Kato fan of $\cZ^+$ under the local isomorphism (\ref{equ local isomo kato fans}).
\begin{lemma} \label{lemma number of preimages} If $x'$ is in the closure of $x$, and $y'$ in the closure of $y$, then $n(x',y')\geqslant n(x,y)$.
\end{lemma}
\begin{proof}
Let $z'$ be a preimage of the pair $(x',y')$. By Proposition \ref{local isomo kato fan product}, there exists an open neighbourhood $U_{z'}$ of $z'$ such that the restriction of $F_{\cZ} \rightarrow F_{\cX} \times_{F_\mathscr{T}} F_{\cY}$ to $U_{z'}$ is an isomorphism onto its image. In particular, $(x,y)$ lies in this image. Thus, there exists a unique preimage of $(x,y)$ that is contained in $U_{z'}$. It follows that $n(x',y')\geqslant n(x,y)$.
\end{proof}
}

\subsection{Semistability and Kato fans associated to the fibred products.} 
\spa{We recall that a log-regular log scheme $\cX^+$ is said to be semistable if the special fiber is reduced. We will see that semistability is a sufficient condition to establish injectivitiy of the local isomorphism (\ref{equ local isomo kato fans}). }

\spa{We recall that, locally around a point $x$ of $\cX_k$, the morphism of characteristic monoids $\N \rightarrow \caC_{\cX,x}$ is a saturated morphism if, for any morphism $u: \N \rightarrow P$ of $fs$ monoids, the amalgamated sum $\caC_{\cX,x} \oplus_{\N} P$ is still saturated.

Following the work by T. Tsuji in an unpublished 1997 preprint, Vidal in \cite{Vidal} defines the saturation index of a morphism of \emph{fs} monoids. In the case of log-regular log scheme over $S^+$ it can be easily computed at point in the special fiber: it is the least common multiple of the multiplicities of the prime components of $\cX_k$. The following criterion holds.
\begin{lemma} \label{lemma criterion saturation}(\cite{Vidal}, Section 1.3)
A morphism of fs monoids is saturated if and only if the saturation index is equal to 1.
\end{lemma}
}

\begin{prop} \label{prop semistability and iso Kato fans}
Assume that the residue field $k$ is algebraically closed. Let $\cX^+$ and $\cY^+$ be log-smooth log scheme over $S^+$. Let $\cZ^+$ be their $fs$ fibred product. If $\cX^+$ is semistable, then for any pair of points $(x,y)$ in $F_\cX \times_{F_S} F_\cY$ whose closures intersect the special fibers $\cX_k$ and $\cY_k$ respectively, the morphism $F_{\cZ}\xrightarrow{} F_{\cX} \times_{F_S} F_{\cY},$ induced by the projections $\cZ^+ \rightarrow \cX^+$ and $\cZ^+ \rightarrow \cY^+$, is a bijection above the pair $(x,y)$, namely $n(x,y)=1$.
\end{prop}

\begin{proof}
By hypothesis $\cX^+$ is a semistable log-regular log scheme over $S^+$, hence the morphism $\N \rightarrow \caC_{\cX,x}$ has saturation index $1$ around any point $x \in F_\cX \cap \cX_k$, so by Lemma \ref{lemma criterion saturation} it is a saturated morphism of characteristic monoids. The saturation condition implies that the fibred product in the category of log schemes coincides with the fibred product in the category of $fs$ log schemes. In particular, the underlying scheme of the special fiber of $\cZ^+$ coincides with the usual schematic fibred product, hence its points are characterized as follows:$$z=(x,y,s,\fp) \text{ and } \caO_{\cZ,z}=(\caO_{\cX,x} \otimes_R \caO_{\cY,y})_{\fp}$$ where $x$ and $y$ are points of $\cX^+$ and $\cY^+$ both mapped to $s$, while $\fp$ is a prime ideal of the tensor product of residue fields $\kappa(x) \otimes_{k} \kappa(y)$. We look for a characterization of points $z$ in $\cZ^+$ that lie in the Kato fan $F_{\cZ}$.

If the point $z$ lies in $F_{\cZ}$, then the maximal ideal $\mathfrak{m}_z$ is equal to the ideal $\caI_{\cZ,z}$ by definition. By the flatness of the models $\cX^+$ and $\cY^+$ over $S^+$, the morphisms of local rings $\caO_{\cX,x} \rightarrow \caO_{\cZ,z}$ and $\caO_{\cY,y} \rightarrow \caO_{\cZ,z}$ are injective. Hence, the equalities $\mathfrak{m}_x= \caI_{\cX,x}$ and $\mathfrak{m}_y= \caI_{\cY,y}$ hold. Thus, the points $z$ in $\cZ^+$ that lie in the Kato fan $F_{\cZ}$ are necessarily points such that the projections $x$ and $y$ to $\cX^+$ and $\cY^+$ lie in the respective associated Kato fans. Therefore, we may assume $x \in F_{\cX}$, $y \in F_{\cY}$, and it remains to characterize the prime ideals $\fp$ such that $z=(x,y,s,\fp) \in F_{\cZ}$.

By log-regularity of $\cZ^+$, the point $z$ lies in the associated Kato fan if and only if $\dim \caO_{\cZ,z}= \rank \caC_{\cZ,z}^{\text{gp}}$. At the level of characteristic sheaves it holds that $$\rank \caC_{\cZ,z}^{\text{gp}}= \rank \caC_{\cX,x}^{\text{gp}} + \rank \caC_{\cY,y}^{\text{gp}} - 1.$$ Since $x$ and $y$ are both assumed to be points in the associated Kato fans, the equality between dimension of local rings and rank of the groupifications of characteristic sheaves lead to the equivalence \begin{align*}
z \in F_{\cZ}  \Leftrightarrow \dim \caO_{\cZ,z} & = \rank \caC_{\cZ,z}^{\text{gp}} = \rank \caC_{\cX,x}^{\text{gp}} + \rank \caC_{\cY,y}^{\text{gp}} - 1 \\
& = \dim \caO_{\cX,x} +  \dim \caO_{\cY,y} -1.
\end{align*}
By log-regularity of $\cZ^+$, it holds that $\dim \caO_{\cZ,z} \geqslant \rank \caC_{\cZ,z}^{\text{gp}}$, thus the inequality $$\dim \caO_{\cZ,z} \geqslant \dim \caO_{\cX,x} +  \dim \caO_{\cY,y} -1$$ is always true and equality holds only for minimal prime ideals $\fp$ of $\kappa(x) \otimes_{k} \kappa(y)$. Therefore, in order to determine the number $n(x,y)$ of preimages of $(x,y)$ in $F_\cZ$, we need to study the number of minimal prime ideals of the tensor product $\kappa(x) \otimes_{k} \kappa(y)$. By assumption, the residue field $k$ is algebraically closed field. It follows that the tensor product $\kappa(x) \otimes_{k} \kappa(y)$ is a domanin, hence it has a unique minimal prime ideal, namely $0$. We obtain that $n(x,y)=1$ for any pair of points of $F_\cX \times_{F_S} F_\cY$ that lie in the special fibers.

Let $(x,y)$ be a pair of points in $F_\cX \times_{F_S} F_\cY$ whose closures intersect the special fibers, namely there exist $x' \in F_\cX \cap \cX_k$ and $y' \in F_\cY \cap \cY_k$ such that $x'$ is in the closure of $x$ and $y'$ in the closure of $y$. Then, by the previous part of the proof and by Lemma \ref{lemma number of preimages}, we have $n(x,y) \leqslant n(x',y') =1$. 
\end{proof}

\section{The skeleton of a log-regular log scheme} \label{sect skeleton log}
\subsection{Construction of the skeleton of a log-regular log scheme.}

\spa{Let $\cX^+$ be a log-regular log scheme over $S^+$. Let $x$ be a point of the associated Kato fan $F$. Denote by $F(x)$ the set of points
$y$ of $F$ such that $x$ lies in the closure of $\{y\}$, and by
$\mathcal{C}_{F(x)}$ the restriction of $\mathcal{C}_F$ to $F(x)$.
Denote
 by $\Spec \mathcal{C}_{\cX,x}$ the spectrum of the monoid
$\mathcal{C}_{\cX,x}=\mathcal{C}_{F,x}$. Then there exists a
canonical isomorphism of
 monoid spaces
$$(F(x),\mathcal{C}_{F(x)})\to \Spec \mathcal{C}_{\cX,x}:y\mapsto \{s\in \mathcal{C}_{\cX,x}\,|\,s(y)= 0\}$$ where the expression $s(y)= 0$ means that $s'(y)= 0$ for any representative $s'$ of $s$ in $\mathcal{M}_{\cX,x}$. In particular, we obtain a bijective correspondence between the faces of the monoid $\mathcal{C}_{\cX,x}$ and the points of $F(x)$, and for every point $y$ of $F(x)$, a surjective cospecialization morphism of monoids $$\tau_{x,y}:\mathcal{C}_{\cX,x}\to \mathcal{C}_{\cX,y}$$ which induces an isomorphism of monoids $$S^{-1}\mathcal{C}_{\cX,x}/(S^{-1}\mathcal{C}_{\cX,x})^{\times}\cong \mathcal{C}_{\cX,x}/S \xrightarrow{\sim}
 \mathcal{C}_{\cX,y}$$ where $S$ denotes the monoid of elements $s$ in $\mathcal{C}_{\cX,x}$ such that $s(y)\neq 0$.}

\spa{For each point $x$ in $F$, we denote by $\sigma_x$ the set of morphisms of monoids $$\alpha:\mathcal{C}_{\cX,x}\to (\R_{\geqslant 0},+)$$ such that $\alpha(\pi)=1$ for every uniformizer $\pi$ in $R$. We endow $\sigma_x$ with the topology of pointwise covergence, where $\R_{\geqslant 0}$ carries the usual Euclidean topology. Note that $\sigma_x$ is a polyhedron, possibly unbounded, in the real affine space
$$\{\alpha:\mathcal{C}^{\gp}_{\cX,x}\to (\R,+)\,|\,\alpha(\pi)=1\mbox{ for every uniformizer }\pi\mbox{ in
}R\}.$$ If $y$ is a point of $F(x)$, then the surjective
cospecialization morphism $\tau_{x,y}$ induces a topological
embedding $\sigma_y\to \sigma_x$ that identifies $\sigma_y$ with a face of $\sigma_x$.}

\spa{We denote by $T$ the disjoint union of the topological
spaces $\sigma_x$ with $x$ in $F$. On the topological space $T$, we consider the equivalence relation $\sim$ generated by couples of the form $(\alpha,\alpha\circ \tau_{x,y})$ where $x$ and $y$ are points in $F$ such that $x$ lies in the closure of $\{y\}$ and $\alpha$ is a point of $\sigma_y$.

The skeleton of $\cX^+$ is defined as the quotient of the topological space $T$ by the equivalence relation $\sim$. We denote this skeleton by $\Sk(\cX^+)$. It is clear that $\Sk(\cX^+)$ has the structure of a polyhedral complex with cells $\{\sigma_x,\,x\in F\}$, so it comes equipped with a piecewise linear (PL) structure, and that the faces of a cell $\sigma_x$ are precisely the cells $\sigma_y$ with $y$ in $F(x)$. 
}

\spa{ \label{rem points in kato fan and skeleton}
We note that $\sigma_x$ is empty for any point $x$ that does not lie in the special fiber $\cX_k$: indeed, outside the special fiber any uniformizer is an invertible element, so it is trivial in $\caC_{\cX,x}$ and is mapped to $0$ by any morphism of monoids. Therefore, the construction of the skeleton associated to $\cX^+$ only concerns the points in the Kato fans $F$ that lie in the special fiber. In particular, given a generic point $x \in \cX_k$ of an intersection of components of $D_\cX$, where at least one component is not in the special fiber, the corresponding face $\sigma_x$ is unbounded.

In other words, the skeleton associated to a log-regular scheme $\cX^+$, where $D_\cX$ allows horizontal components, generalizes Berkovich's skeletons by admitting unbounded faces in the direction of the horizontal components as well as by allowing singularities in the special fiber. It also generalizes the construction performed by Gubler, Rabinoff and Werner in \cite{GublerRabinoffWerner} of a skeleton associated to a strictly semistable snc pair.
}
 
\subsection{Embedding the skeleton in the non-archimedean generic fiber.}
\spa{Let $\cX^+$ be a log-regular log scheme over $S^+$. Let $x$ be a point of the associated Kato fan $F$. As the log structure on $\cX^+$ is of finite type, the characteristic monoid $\caC_{\cX,x}$ is of finite type too, and thus $\caC_{\cX,x}^\gp$ is a free abelian group of finite rank. Hence there exists a section $$\zeta: \caM_{\cX,x}^\gp / \caM_{\cX,x}^\times \rightarrow \caM_{\cX,x}^\gp.$$ The section $\zeta$ restricts to $\caC_{\cX,x} \rightarrow \caM_{\cX,x}$; indeed, if $x \in \caM_{\cX,x}$ then $\zeta(\overline{x})-x \in \caM_{\cX,x}^\times$. Therefore we may choose a section \begin{equation}\label{eq:sect}\mathcal{C}_{\cX,x}\to
\mathcal{M}_{\cX,x}\end{equation} of the projection homomorphism
$$\mathcal{M}_{\cX,x}\to \mathcal{C}_{\cX,x}$$ and use this section to view $\mathcal{C}_{\cX,x}$ as a submonoid of $\mathcal{M}_{\cX,x}$. Note that $\mathcal{C}_{\cX,x}\setminus \{0\}$ generates the ideal $\caI_{\cX,x}$ of $\mathcal{O}_{\cX,x}$.
\\

We propose a generalisation of \cite{MustataNicaise}, Lemma 2.4.4. \begin{lemma} \label{lemma admissible expansion} Let $A$ be a Noetherian ring, let $I$ be an ideal of $A$ and let $(y_1,\ldots,y_m)$ be a system of generators for $I$. We denote by $\hat{A}$ the $I$-adic completion of $A$. Let $B$ be a subring of $A$ such that the elements $y_1,\ldots,y_m$ belong to $B$ and generate the ideal $B \cap I$ in $B$. Then, in the ring $\hat{A}$, every element $f$ of $B$ can be written as \begin{equation} \label{equ lemma admiss form}
f=\sum_{\beta \in \Z^m_{\geqslant 0}} c_\beta y^\beta
\end{equation} where the coefficients $c_\beta$ belong to $((A\setminus I) \cap B) \cup \{0\}$.
\end{lemma} 
\begin{proof}
Let $f$ be an element of $B$, we construct an expansion for $f$ of the form (\ref{equ lemma admiss form}) by induction. If $f$ belongs to the complement of $I$, the conclusion trivially holds. Otherwise, $f$ belongs to $I$ and we can write $f$ as a linear combination of the elements $y_1,\ldots,y_m$ with coefficients in $B$: $$f= \sum_{j=1}^{m} b_j y_j, \quad b_j \in B.$$

By induction hypothesis, we suppose that $i$ is a positive integer and that we can write every $f$ in $B$ as a sum of an element $f_i$ of the form (\ref{equ lemma admiss form}) and a linear combination of degree $i$ monomials in the elements $y_1,\ldots, y_m$ with coefficients in $B$. We apply this assumption to the coefficients $b_j$, hence
$$b_j=b_{j,i} + \sum_{\substack{\beta \in \Z^m_{\geqslant 0} \\|\beta|=i}} b_{j,\beta} y^\beta, \quad b_{j, \beta} \in B.$$ Then we can write $f$ as a sum of an element $f_{i+1}$ of the form (\ref{equ lemma admiss form}) and a linear combination of degree $i+1$ monomials in the elements $y_1,\ldots,y_m$ with coefficients in $B$
$$ f=\underbrace{\sum_{j=1}^{m} b_{j,i} y_j}_{f_{i+1}} + \sum_{j=1}^{m}\Big( \sum_{\substack{\beta \in \Z^m_{\geqslant 0} \\|\beta|=i}} b_{j,\beta} y^\beta \Big) y_j$$ such that $f_i$ and $f_{i+1}$ have the same coefficients in degree less or equal to $i$. Iterating this construction we finally find an expansion of $f$ of the required form.
\end{proof}}

\spa{ \label{paragr admissible expansion}Let $f$ be an element of $\mathcal{O}_{\cX,x}$. Considering $A=B=\caO_{\cX,x}$, $I=\mathfrak{m}_x$ and a system of generators for $\mathfrak{m}_x$ in $\caC_{\cX,x} \setminus \{1\}$, by Lemma \ref{lemma admissible expansion} we can write $f$ as a formal power series
\begin{equation}\label{eq-adm}
f=\sum_{\gamma \in \mathcal{C}_{\cX,x}}c_{\gamma}\gamma
\end{equation}
 in
$\widehat{\mathcal{O}}_{\cX,x}$, where each coefficient $c_\gamma$
 is either zero or a unit in $\mathcal{O}_{\cX,x}$. We call this formal series an \textit{admissible expansion} of $f$. We set
 \begin{equation}\label{equ def S}
S=\{\gamma\in\mathcal{C}_{\cX,x}\,|\,c_\gamma\neq 0\}
 \end{equation} and we denote by $\Gamma_x(f)$ the set of elements of $S$ that lie on a compact face of the convex hull of $S+ \mathcal{C}_{\cX,x}$ in $\mathcal{C}^{\gp}_{\cX,x}\otimes_{\Z}\R$. We call $\Gamma_x(f)$ the {\em initial support} of $f$ at $x$, notation which is justified by the next proposition.}
\begin{prop}\label{prop-init}\item
\begin{enumerate}
\item \label{it:indep1} The element
$$f_x=\sum_{\gamma\in\Gamma_x(f)}c_\gamma(x) \gamma\quad \in k(x)[\mathcal{C}_{\cX,x}]$$ depends on the choice of the section \eqref{eq:sect}, but not on the expansion \eqref{eq-adm}. \item \label{it:indep2} The subset $\Gamma_x(f)$ of $\mathcal{C}_{\cX,x}$ only depends on $f$ and $x$, and not on the choice of the section \eqref{eq:sect} or the expansion \eqref{eq-adm}.
\end{enumerate}
\end{prop}
\begin{proof}
If we denote by $I$ the ideal of $k(x)[\mathcal{C}_{\cX,x}]$ generated by $\mathcal{C}_{\cX,x}\setminus \{1\}$, then it follows from \cite{Kato1994a} that there exists an isomorphism of $k(x)$-algebras \begin{equation} \label{equ isomo graded algebras}
\mathrm{gr}_I \,k(x)[\mathcal{C}_{\cX,x}]\to \mathrm{gr}_{\mathfrak{m}_x}\, \mathcal{O}_{\cX,x}.
\end{equation}
Using this result and following the argument of  \cite{MustataNicaise} Proposition 2.4.4, we show now that $f_x$ does not depend on the expansion of $f$. Let $$f=\sum_{\gamma \in \mathcal{C}_{\cX,x}}c'_{\gamma}\gamma$$ be another admissible expansion of $f$ with associated set $\Gamma_x(f)'$ and element $f'_x$. Then $$0=\sum_{\gamma \in \mathcal{C}_{\cX,x}}(c_{\gamma}- c'_\gamma)\gamma =\sum_{\gamma \in \mathcal{C}_{\cX,x}}d_{\gamma}\gamma$$ where the right hand side is an admissible expansion obtained by choosing admissible expansions for the elements $c_\gamma-c'_\gamma$ that do not lie in $\caO_{\cX,x}^\times \cup \{0\}$. In particular $d_\gamma(x)= c_\gamma(x)-c'_\gamma(x)$ for any $\gamma$ in $ \Gamma_x(f) \cup\Gamma_x(f)'$. The isomorphism of graded algebras in (\ref{equ isomo graded algebras}) implies that the elements $d_\gamma$ must all vanish, hence $\Gamma_x(f) = \Gamma_x(f)'$ and $f_x=f'_x$.  

Point \eqref{it:indep2} follows from the fact that the coefficients $c_\gamma$ of $f_x$ are independent of the chosen section up to multiplication by a unit in $\mathcal{O}_{\cX,x}$, so that the support $\Gamma_x(f)$ of $f_x$ only depends on $f$ and $x$.
\end{proof}

\begin{prop}\label{prop-val}
Let $x$ be a point of $F$ and let $$\alpha:\mathcal{C}_{\cX,x}\to (\R_{\geqslant 0},+)$$ be an element of $\sigma_x$.Then there exists a unique minimal real valuation $$v:\mathcal{O}_{\cX,x} \setminus \{0\}\to \R_{\geqslant 0}$$ such that $v(m)=\alpha(\overline{m})$ for each element $m$ of $\mathcal{M}_{\cX,x}$.
\end{prop}
\begin{proof}
We will prove that the map
\begin{equation}\label{eq:val}
v:\mathcal{O}_{\cX,x} \setminus \{0\}\to \R:f\mapsto
\min\{\alpha(\gamma)\,|\,\gamma\in \Gamma_x(f)\}\end{equation}
satisfies the requirements in the statement. We fix a section $$\mathcal{C}_{\cX,x}\to \mathcal{M}_{\cX,x}.$$
It is straightforward to check that $(f\cdot g)_x=f_x\cdot g_x$ for all $f$ and $g$ in $\mathcal{O}_{\cX,x}$. This implies that $v$ is a valuation. It is obvious that
$v(m)=\alpha(\overline{m})$ for all $m$ in $\mathcal{M}_{\cX,x}$, since we can write $m$ as the product of an element of $\mathcal{C}_{\cX,x}$ and a unit in $\mathcal{O}_{\cX,x}$.

Now we prove minimality. Consider any real valuation $$w:\mathcal{O}_{\cX,x}\to \R$$ such that $w(m)=\alpha(\overline{m})$ for each element $m$ of
$\mathcal{M}_{\cX,x}$, and let $f$ be an element of
$\mathcal{O}_{\cX,x}$. We must show that $w(f)\geqslant v(f)$.

We set $$C_{\alpha}=\mathcal{C}_{\cX,x}\setminus \alpha^{-1}(0).$$ We denote by $I$ the ideal in $\mathcal{O}_{\cX,x}$ generated by $C_{\alpha}$ and by $A$ the $I$-adic completion of $\mathcal{O}_{\cX,x}$. By Lemma \ref{lemma admissible expansion}, we see that we can write $f$ in $A$ as \begin{equation} \label{equ expasion wrt I}
\sum_{\beta \in C_{\alpha}\cup\{1\}}d_\beta \beta
\end{equation} where $d_\beta$ is either zero or contained in the complement of $I$ in $\mathcal{O}_{\cX,x}$.

Since $\alpha(\beta)>0$ for every $\beta\in C_{\alpha}$, we can find an integer $N>0$ such that $w(g)>w(f)$ for every element $g$ in  $I^N$. So we have $$w(f)\geqslant \min \{\alpha(\beta)\,|\,d_\beta \neq 0\}$$ recalling that $w(\beta)=\alpha(\beta)$ for all $\beta$ in $\mathcal{C}_{\cX,x}$.

We consider the coefficients in the expansion (\ref{equ expasion wrt I}) of $f$. Applying Lemma \ref{lemma admissible expansion} as in paragraph (\ref{paragr admissible expansion}), we can write admissible expansions of these coefficients in $\widehat{\mathcal{O}}_{\cX,x}$ as $$d_\beta = \sum_{\gamma \in \caC_{\cX,x}}c_{\gamma, \beta} \gamma, \quad c_{\gamma, \beta} \in \caO_{\cX,x}^\times \cup \{0\},$$ with $\alpha(\gamma)=0$ in the expansions of $d_\beta$ that belong to $\mathfrak{m}_x \setminus I$. 

Therefore we obtain an admissible expansion of $f$ $$f = \sum_{\substack{\beta \in C_\alpha \cup \{1\} \\ \gamma \in \caC_{\cX,x}}}c_{\gamma, \beta}\, \gamma\beta $$ and we have $v(f)=\min\{\alpha(\gamma \beta)\,|\,c_{\gamma, \beta}\neq 0\} \geqslant\min\{\alpha(\beta)\,|\,d_{\beta}\neq 0\} \geqslant w(f).$
\end{proof}
 
\begin{rem} \label{rem def-val}
In the definition (\ref{eq:val}) of the valuation $v$, we compute the minimum over the terms in the initial support of $f$: these elements are a finite number and they only depends on $x$ and $f$ by Proposition \ref{prop-init}. Therefore, this minimum provides a well-defined function on $\caO_{\cX,x} \setminus \{0\}$. Nevertheless, it is equivalent to consider the minimum over all the terms of an admissible expansion of $f$, i.e. for any admissible expansion $f=\sum_{\gamma \in \mathcal{C}_{\cX,x}}c_{\gamma}\gamma$ $$\min\{\alpha(\gamma)\,|\,\gamma\in \Gamma_x(f)\} = \min\{\alpha(\gamma)\,|\,\gamma\in S\},$$ where $S=\{\gamma\in\mathcal{C}_{\cX,x}\,|\,c_\gamma\neq 0\}$ as in (\ref{equ def S}). Indeed, any element that belongs to $S$ can be written as a sum of an element of the initial support of $f$ and an element of $\caC_{\cX,x}$. Since the morphism $\alpha$ is additive and takes positive real values, then the minimum is necessarily attained by the elements in the initial support.
\end{rem}

\spa{ We will denote the valuation $v$ from Proposition
\ref{prop-val} by $v_{x,\alpha}$. Since $v_{x,\alpha}$ induces a real valuation on the function field of $\cX_K$ that extends the discrete valuation $v_K$ on $K$, it defines a point of the $K$-analytic space $\cX_K^\an$, which we will denote by the same symbol $v_{x,\alpha}$. We now show that the characterization of $v_{x,\alpha}$ in Proposition \ref{prop-val} implies that $$v_{y,\alpha'}=v_{x,\alpha'\circ \tau_{x,y}}$$ for every $y$ in $F(x)$ and every $\alpha'$ in $\sigma_y$.

Firstly we note that  $\mathcal{O}_{\cX,y}$ is the localization of $\mathcal{O}_{\cX,x}$ with respect to the elements of $m\in \mathcal{M}_{\cX,x}$ in the kernel of $\tau_{x,y}$.
Indeed, by construction of $\tau_{x,y}$, the kernel is given by $$\ker (\tau_{x,y})= \{s \in \caC_{\cX,x} | s(y) \neq 0 \};$$ to obtain $\mathcal{O}_{\cX,y}$ from $\mathcal{O}_{\cX,x}$, we localize  by $$S=\{a \in \mathcal{O}_{\cX,x} | a(y) \neq 0 \};$$ therefore we can identify the set of elements of $\caM_{\cX,x}$ whose reduction is in $\ker(\tau_{x,y})$ with the set $S$, recalling that, for points in the Kato fan, $\mathcal{C}_{\cX,x}\setminus \{1\}$ generates the maximal ideal of $\mathcal{O}_{\cX,x}$. Therefore we are dealing with these two morphisms:
$$\caO_{\cX,x} \hookrightarrow S^{-1}\caO_{\cX,x}=\caO_{\cX,y},$$
$$ \caC_{\cX,x} \twoheadrightarrow \caC_{\cX,x}/S = \caC_{\cX,y}.$$

Let $f$ be an element of $ \caO_{\cX,x}$. Under the notations of Lemma \ref{lemma admissible expansion}, we apply the lemma to $A= \caO_{\cX,y}$ and $B=\caO_{\cX,x}$, choosing a system of generators of $\mathfrak{m}_y$ in $\caC_{\cX,x}$: we can find an admissible expansion of $f$ of the form $$f=\sum_{\delta \in \mathcal{C}_{\cX,y}}d_{\delta}\delta \quad \text{ with }d_\delta \in (\caO_{\cX,x} \cap \caO_{\cX,y}^\times) \cup \{0\}.$$ Admissible expansions of coefficients $d_\delta$ induce an admissible expansion for $f$ by $$f=\sum_{\delta \in \mathcal{C}_{\cX,y}}\big( \sum_{\gamma \in S} c_{\gamma \delta} \gamma \big)\delta \quad \text{ with }c_{\gamma \delta} \in \caO_{\cX,x}^\times \cup \{0\},$$ where $\gamma$ runs through the set $S$ since $d_\delta \in \caO_{\cX,y}^\times$. Thus we have \begin{align*}
v_{y,\alpha'}(f) &= \min\{\alpha'(\delta)\,|\,\delta\in \Gamma_y(f)\}\\
& = \min\{\alpha' \circ \tau_{x,y}( \gamma \delta)\,|\,\delta\in \Gamma_y(f), \gamma \in S\}\\
& = \min\{\alpha' \circ \tau_{x,y}( \gamma \delta)\,|\,\gamma\delta\in \Gamma_x(f)\}\\
& = v_{x,\alpha' \circ \tau_{x,y}}(f).
\end{align*}
Hence, we obtain a well-defined map $$\iota:\Sk(\cX^+)\to \cX_K^\an$$ by sending $\alpha$ to $v_{x,\alpha}$ for every point $x$ of $F$ and every $\alpha\in \sigma_x$.}

\begin{prop}\label{prop-embed}
The map $$\iota:\Sk(\cX^+)\to\cX_K^\an$$  is a
topological embedding.
\end{prop}
\begin{proof}
First, we show that $\iota$ is injective. Let $x$ be a point of $F$ and $\alpha$ an element of $\sigma_x$. Let $y$ be the point of $F(x)$ corresponding to the face $\mathcal{C}_{\cX,x}\setminus
\alpha^{-1}(0)$ of $\mathcal{C}_{\cX,x}$. Then $\alpha$ factors through an element
$$\alpha':\mathcal{C}_{\cX,y}\to \R_{\geqslant 0}$$ of $\sigma_y$.  Note that $\alpha=\alpha'$ in $\Sk(\cX^+)$ because $\alpha=\alpha'\circ \tau_{x,y}$.
Moreover, since $(\alpha')^{-1}(0)=\{1\}$, the center of the valuation $v_{y,\alpha'}$ is the point $y$, so that $\redu_{\cX}(v_{y,\alpha'})=y$. Thus we can recover $y$ from $v_{y,\alpha'}$. Then we can also reconstruct $\alpha'$ by looking at the values of $v_{y,\alpha'}$ at the elements of $\mathcal{M}_{\cX,y}$. We conclude that $\iota$ is injective.

Now, we show that $\iota$ is a homeomorphism onto its image. For any valuation $v$ in $\Sk(\cX^+)$ and any small open neighbourhood $U$ of $\iota(v)$ in $\cX_K^\an$, there exists a closed subset $C$ in $\Sk(\cX^+)$ such that $U \cap \iota(\Sk(\cX^+))\subseteq \iota(C)$ and, up to subdivisions, we can assume that the $C$ is a closed cell of $\Sk(\cX^+)$. Therefore, it suffices to prove that the restriction of $\iota$ to any closed cell $\sigma_x$ of $\Sk(\cX^+)$ is an homeomorphism. The restriction $\iota_{|\sigma_x}$ is an injective map from a compact set to the Hausdorff space $\cX_K^\an$, so we reduce to show that $\iota_{|\sigma_x}$ is continuous, to conclude that $\iota_{|\sigma_x}$ is a homeomorphism. By definition of the Berkovich topology, it is enough to prove that the map $$\sigma_x\to \R:\alpha\mapsto v_{x,\alpha}(f)$$ is continuous for every $f$ in $\mathcal{O}_{\cX,x}$. This is obvious from the formula \eqref{eq:val}. 
\end{proof}

\spa{From now on, we will view $\Sk(\cX^+)$ as a topological subspace of $\cX_K^{\an}$ by means of the embedding $\iota$ in Proposition \ref{prop-embed}. If $\cX$ is regular over $R$ and $\cX_k$ is a divisor with strict normal crossings, the skeleton $\Sk(\cX^+)$ was described in \cite{MustataNicaise}, Section 3.1.}

\subsection{Contracting the generic fiber to the skeleton.}

\spa{We denote by $D_{\cX,\text{hor}}$ the component of $D_\cX$ not contained in the special fiber $\cX_k$. The inclusion $\iota: \Sk(\cX^+) \rightarrow \cX_K^\an$ is actually an inclusion in $(\cX_K \setminus D_{\cX, \text{hor}})^\an$ and it admits a continuous retraction $$\rho_\cX: (\cX_K \setminus D_{\cX, \text{hor}})^\an\rightarrow \Sk(\cX^+)$$ constructed as follows. Let $x$ be a point of $(\cX_K \setminus D_{\cX, \text{hor}})^\an$ and consider the reduction map $$\redu_{\cX}:(\cX_K \setminus D_{\cX, \text{hor}})^\an \rightarrow \cX_k.$$ Let $E_1, \ldots, E_r$ be the irreducible components of $D_\cX$ passing through the point $\redu_\cX(x)$. We denote by $\xi$ the generic point of the connected component of $E_1 \cap \ldots \cap E_r$ that contains $\redu_\cX(x)$. By Lemma \ref{lemma points Kato fan}, $\xi$ is a point in the associated Kato fan $F$. We set $\alpha$ to be the morphism of monoids $$\alpha: \caC_{\cX,\xi} \rightarrow \R_{\geqslant 0 }$$ such that $\alpha(\overline{m}) = v_x(m)$ for any element $m$ of $\caM_{\cX,\xi}$. In particular $\alpha(\pi)= v_x(\pi)=1$ as we assumed the normalization of all valuations in the Berkovich space. Then $\rho_\cX(x)$ is the point of $\Sk(\cX^+)$ corresponding to the couple $(\xi,\alpha)$. By construction $\rho_\cX$ is continuous and right inverse to the inclusion $\iota$. 
}
\spa{Given a dominant morphism $f: \cX^+ \rightarrow \cY^+$ of integral flat separated log-regular log schemes over $S$, it induces a map between the set of birational points $\text{Bir}(\cX_K) \rightarrow \text{Bir}(\cY_K)$. As $\text{Bir}(\cX_K) \subseteq  (\cX_K \setminus D_{\cX, \text{hor}})^\an$, we can employ the retraction $\rho$ to define a map of skeletons as follows
\begin{center}
$\begin{gathered}
\xymatrixcolsep{2pc} \xymatrix{
 \text{Bir}(\cX_K) \ar[r]^-{\widehat{f}} \ar[d]_{\rho_\cX} &  \text{Bir}(\cY_K) \ar[d]^{\rho_\cY}\\
  \Sk(\cX^+) \ar@{.>}[r]^{} \ar@/_/[u]_-{\iota_\cY}  & \Sk(\cY^+).
}
\end{gathered}$\end{center} This association makes the skeleton construction $\Sk(\cX^+)$ functorial in $\cX^+$ with respect to dominant morphisms.
}

\subsection{Skeleton of a $fs$ fibred product.}
\spa{Let $\cX^+$ and $\cY^+$ be log-smooth log schemes over $S^+$, let $\cZ^+$ be their $fs$ fibred product. Let $$\Sk(\cZ^+) \rightarrow \Sk(\cX^+) \times \Sk(\cY^+)$$ be the continuous map of skeletons functorially associated to the projections $\pr_{\cX}:\cZ^+ \rightarrow \cX^+$ and $\pr_{\cY}:\cZ^+ \rightarrow \cY^+$. We denote this map by $\big(\pr_{\Sk(\cX)}, \pr_{\Sk(\cY)}\big)$ and we recall that it is constructed considering the diagram
\begin{equation} \label{diagram map on skeletons funct induced}
\begin{gathered}
\xymatrixcolsep{6pc} \xymatrix{
 \text{Bir}(\cZ_K) \ar[r]^-{(\widehat{\pr_{\cX}}, \widehat{\pr_{\cY}})} \ar[d]_{\rho_\cZ} & \text{Bir}(\cX_K) \times  \text{Bir}(\cY_K) \ar[d]^{(\rho_\cX, \rho_\cX)}\\
  \Sk(\cZ^+) \ar[r]^-{(\pr_{\Sk(\cX)}, \pr_{\Sk(\cY)})} \ar@/_/[u]_-{\iota_\cZ}  & \Sk(\cX^+) \times \Sk(\cY^+).
}
\end{gathered}\end{equation}
}

\begin{prop}  \label{prop semistability and skeletons}
Assume that the residue field $k$ is algebraically closed.  If $\cX^+$ is semistable, then the map $\big(\pr_{\Sk(\cX)}, \pr_{\Sk(\cY)}\big)$ is a PL homeomorphism.
\end{prop}
\begin{proof}
The surjectivity of the map $\big(\pr_{\Sk(\cX)}, \pr_{\Sk(\cY)}\big)$ follows from the commutativity of the diagram (\ref{diagram map on skeletons funct induced}) and the surjectivity of $(\rho_\cX, \rho_\cX) \circ (\widehat{\pr_{\cX}}, \widehat{\pr_{\cY}})$. To prove the injectivity of $\big(\pr_{\Sk(\cX)}, \pr_{\Sk(\cY)}\big)$, we provide an explicit description of the map $\pr_{\Sk(\cX)}$.

We recall that the projection $\widehat{\pr_{\cX}}$ is such that a valuation $v$ on the function field $K(\cZ_K)$ maps to the composition $v \circ i$ where $i:K(\cX_K) \hookrightarrow K(\cZ_K)$.

Let $v_{z,\varepsilon}$ be the valuation in $\Sk(\cZ^+)$ corresponding to a couple $(z,\varepsilon)$ with $z \in F_{\cZ} \cap \cZ_k$ and $\varepsilon \in \sigma_z$. We consider the morphism of associated Kato fans $$F_{\cZ} \rightarrow F_{\cX} \times_{F_S} F_{\cY}$$ as established in Proposition \ref{local isomo kato fan product}. We denote respectively by $\pr_{F_\cX}$ and $\pr_{F_\cY}$ the projection to the first and second factor. Then $\pr_{F_{\cX}}(z)$ is a point in the associated Kato fan $F_{\cX}$, that we denote by $x$. We consider the morphism of monoids $$i_x: \caC_{\cX,x} \rightarrow \caC_{\cZ,z}$$ and the composition \begin{center}
$\begin{gathered}
\xymatrixcolsep{1pc} \xymatrixrowsep{0pc} \xymatrix{
    \pr_{\cX}(\varepsilon): & \caC_{\cX,x} \ar[r]^-{i_x} & \caC_{\cZ,z}= (\caC_{\cX,x} \oplus_{\N} \caC_{\cY,y})^{\text{sat}} \ar[r]^-{\varepsilon} & \R_{\geqslant 0} \\
   & a \ar@{|->}[r] & [a,1] \ar@{|->}[r] & \varepsilon([a,1]). \\
}
\end{gathered}$\end{center} It trivially satisfies $\varepsilon \circ i_x(\pi)=1$. In order to conclude that it correctly defines a point in the skeleton $\Sk(\cX^+)$, we need to check the compatibility with respect to the equivalence relation $\sim$. Indeed, suppose that $\varepsilon = \varepsilon' \circ \tau_{z,z'}$ for some $z' \in \overline{\{z\}}$. We denote by $x'$ the projection of $z'$ under the local isomorphism of associated Kato fans. The diagram
\begin{center}
$\begin{gathered}
\xymatrixcolsep{2pc} \xymatrixrowsep{0pc} \xymatrix{
    \caC_{\cX,x} \ar[r]^-{i_x} \ar[dd]^{\tau_{x,x'}} & \caC_{\cZ,z} \ar[rd]^-{\varepsilon} \ar[dd]^{\tau_{z,z'}} & \\
    & & \R_{\geqslant 0} \\
    \caC_{\cX,x'} \ar[r]^-{i_{x'}} & \caC_{\cZ,z'} \ar[ru]_-{\varepsilon'} &  \\
}
\end{gathered}$\end{center} is commutative as made up by a commutative square and a commutative triangle of arrows. Therefore, by commutativity $$\pr_{\cX}(\varepsilon)= \pr_{\cX}(\varepsilon') \circ \tau_{x,x'}$$ and this implies that $\pr_{\cX}(\varepsilon)$ defines a well-defined point $v_{x,\pr_{\cX}(\varepsilon)}$ of $\Sk(\cX^+)$.

We claim that $v_{x, \pr_{\cX}(\varepsilon)}$ is indeed the image of $v_{z,\varepsilon}$ under the map $\pr_{\Sk(\cX)}$, hence that the equality in the following inner diagram holds
\begin{center} 
$\begin{gathered}
\xymatrixcolsep{0.4pc}\xymatrixrowsep{0.1pc} \xymatrix{
 \widehat{\cZ_{\eta}} \ar[rrr]^-{\widehat{\pr_{\cX}}} \ar[dddd]_{\rho_\cZ} & & & \widehat{\cX_{\eta}}\ar[dddd]^{\rho_\cX}\\
 \footnotesize & v_{z,\varepsilon} \ar@{|->}[r] &  v_{z,\varepsilon} \circ i\ar@{|->}[dd] & \\
 & & & \\
 & v_{z,\varepsilon} \ar@{|->}[uu] \ar@{|->}[r] & \rho_\cX(v_{z,\varepsilon} \circ i) = v_{x, \pr_{\cX}(\varepsilon)}& \\
 \normalsize \Sk(\cZ^+) \ar[rrr]_-{\pr_{\Sk(\cX)}} \ar@/_/[uuuu]_-{\iota_\cZ}  & & & \Sk(\cX^+).
}
\end{gathered}$\end{center} We denote $\rho_\cX(v_{z,\varepsilon} \circ i)$ by $(x,\alpha)$ as a point of $\Sk(\cX^+)$. By definition of the retraction $\rho_\cX$, the morphism $\alpha$ is characterized by the fact that $\alpha(\overline{m})= (v_{z,\varepsilon} \circ i )(m)$ for any $m$ in $\caM_{\cX,x}$ and then we have $$\alpha(\overline{m})= (v_{z,\varepsilon} \circ i )(m) = v_{z,\varepsilon}(m)= \varepsilon(\overline{m}).$$ On the other hand, for any $m$ in $\caM_{\cX,x}$ $$v_{x,\pr_{\cX}(\varepsilon)}(m)= \pr_{\cX}(\varepsilon) (\overline{m}) = \varepsilon(\overline{m})$$ hence we obtain that $\alpha$ coincide with the morphism $\pr_{\cX}(\varepsilon)$. It means that their associated points $\rho_\cX(v_{z,\varepsilon} \circ i)$ and $v_{x, \pr_{\cX}(\varepsilon)}$ coincide in $\Sk(\cX^+)$.

Given a pair of points in $\Sk(\cX^+) \times \Sk(\cY^+)$, we know by surjectivity of $\big(\pr_{\Sk(\cX)}, \pr_{\Sk(\cY)}\big)$  that they are of the form $$(v_{x, \pr_{\cX}(\varepsilon)}, v_{y, \pr_{\cY}(\varepsilon)}).$$ The assumptions of semistability of $\cX^+$ and algebraic closedness of $k$ guarantee that there is a unique $z$ in $F_\cZ$ in the fiber of $x$ and $y$, by Proposition \ref{prop semistability and iso Kato fans} and Remark \ref{rem points in kato fan and skeleton}. Moreover, we can uniquely reconstruct $\varepsilon$ by looking at the values of $v_{x, \pr_{\cX}(\varepsilon)}$ at the elements of $\caM_{\cX,x}$ and respectively of $v_{y, \pr_{\cY}(\varepsilon)}$ at the elements of $\caM_{\cY,y}$. We conclude that $\big(\pr_{\Sk(\cX)}, \pr_{\Sk(\cY)}\big)$  is injective.
\end{proof}

The assumption of semistability is crucial in the result of Proposition \ref{prop semistability and skeletons}. To see this, it is helpful to consider an example.
\begin{example}\label{quartic} Let $q$ be the equation a generic quartic curve in $\mathbb{P}^2_{\C((t))}$. Then $\cX:t q+x^2y^2=0$ gives the equation of a family of genus $3$ curves, degenerating to two double lines. The dual complex $\D(\cX_\C)$ of the special fiber $\cX_\C$ is a line segment and $\cX$ has 4 singularities of type $A_1$ in each component of the special fiber, corresponding to the base points of the family. In this case taking a semistable model of $\cX_{\C((t))}$ requires an order two base change $R'$ of $R=\C[[t]]$, which induces coverings branched at each of these singular points (see \cite{HarrisMorrison}, p.133 for details). Let $\cY$ be such a semistable reduction. Thus the special fiber of $\cY$ consists of two elliptic curves, call them $E_1$ and $E_2$, which intersect in two points, $p_A$ and $p_B$, which are the preimages of the point $(0:0:1)$. The dual complex $\D(\cY_\C)$ of the special fiber $\cY_\C$ is isomorphic to $S^1$.

We will compare the dual complex of $(\cX \times_{R} \cX)_\C$ with that of $(\cY\times_{R'} \cY)_\C$. The models $\cX$ and $\cY$ are not log-regular at every point, but from our prospective it is enough that they are log-regular at the generic point of each stratum. For the product with a  semistable model, the dual complex is the product of the dual complexes, and $\D((\cY\times_{R'} \cY)_\C)$ is therefore a real $2$-torus  $S^1 \times S^1$.

On the other hand, the dual complex of the product $(\cX \times_{R} \cX)_\C$ is given by a quotient of $S^1 \times S^1$ by the action of $\mathbb{Z}/2\Z$.  $\D((\cY \times_{R'} \cY)_\C)$ has the structure of a cell complex, whose cells correspond to ordered pairs of strata in  $\D(\cY_\C)$, so
\begin{align*}
&\text{zero-dimensional strata: }&(E_1, E_2), (E_1, E_1), (E_2, E_1), (E_2, E_2)\\
&\text{one-dimensional strata: }&(E_1, p_A), (E_1, p_B), (E_2, p_A), (E_2, p_B)\\
& &(p_A, E_1), (p_B, E_1), (p_A, E_2), (p_B, E_2)\\
&\text{two-dimensional strata: }& (p_A, p_A), (p_A, p_B), (p_B, p_A), (p_B, p_B).\end{align*}
The action of $\mathbb{Z}/2\Z$ fixes $E_1$ and $E_2$, while switching $p_A$ and $p_B$. Therefore it fixes exactly the zero dimensional strata while acting freely on the other points. The quotient, the complex $\D((\cX\times_{R}\cX)_\C)$, is piecewise-linearly homeomorphic to the sphere $S^2$. In particular, it is not isomorphic to the product of two line segments.
\end{example}

\section{The weight function}\label{sec weight}
From now on we assume $\text{char}(K)=\text{char}(k)=0$.
\subsection{Weight function associated to a logarithmic pluricanonical form.}
\spa{Let $X$ be a connected, smooth and proper $K$-variety of dimension $n$. We introduce the following notation: for any log-regular model $\cX^+$ of $X$, for any point $x = (\xi_x,|\cdot|_x) \in X^\an$ and for any $\Q$-Cartier divisor $D$ on $\cX^+$ whose support does not contain $\xi_x$, we set $$v_x(D) = - \ln |f(x)|$$ where $f$ is any element of $K(X)^\times$ such that $D= \divisor(f)$ locally at $\redu_{\cX}(x)$.}

\spa{Let $(X,\Delta_X)$ be an snc pair where $\Delta_X$ is an effective $\Q$-divisor such that $\Delta_X= \sum a_i \Delta_{X,i}$ has $0 \leqslant a_i \leqslant 1$. We denote $X^+=(X,\lceil \Delta_X \rceil)$. Let $\omega$ be a regular $m$-pluricanonical form on $X^+$ with poles of order at most $ma_i$ along $\Delta_{X,i}$, for some $m$ such that $ma_i \in \N$ for any $i$. We call such forms $\Delta_X$-logarithmic $m$-pluricanonical forms.

Moreover, $\omega$ is a regular section of the logarithmic $m$-pluricanonical bundle of $X^+$ and for each log-regular model $\cX^+$ of $X^+$, the form $\omega$ defines a divisor $\divisor_{\cX^+}(\omega)$ on $\cX^+$. Note that the multiplicity in $\divisor_{\cX^+}(\omega)$ of the closure $\overline{\Delta_{X,i}}$ in $\cX$ of $\Delta_{X,i}$ is at least $m(1-a_i)$.
}

\spa{Given a $\Delta_X$-logarithmic $m$-pluricanonical form $\omega$, we can consider it as a rational section of $\omega_{X/K}^{\otimes m}$. Hence, we can associate to $\omega$ the weight function $\weight_{\omega}$ as in \cite{MustataNicaise}. The following gives an interpretation of the weight function associated to $\omega$ in terms of logarithmic differentials, that we will use in the sequel.}
\begin{lemma} \label{lemma log formula for weight function}
Let $\cX^+$ be a log-regular model of $X^+$. Then for every point $x$ of $\Sk(\cX^+)$ $$\weight_{\omega}(x) = v_x(\divisor_{\cX^+}(\omega)) +m.$$
\end{lemma}
\begin{proof}
It suffices to prove the equality for the divisorial points in $\Sk(\cX^+)$, since they are dense in the skeleton, and both $\weight_{\omega}(\cdot)$ and $v_{\cdot}(\divisor_{\cX^+}(\omega))$ are continuous functions on the skeleton $\Sk(\cX^+)$.

Let $x$ be a divisorial point in $\Sk(\cX^+)$. If $x$ corresponds to a component of the special fiber $\cX_k$, then the center $\redu_{\cX}(x)$ of $x$ does not contain the closure of any components of $\Delta_X$. Thus, locally around $\redu_{\cX}(x)$ the log schemes $\cX^+$ and $(\cX,\cX_{k,\redu})$ are isomorphic and in particular $\divisor_{\cX^+}(\omega)=\divisor_{(\cX,\cX_{k,\redu})}(\omega)$. The computation in \cite{NicaiseXu}, Section 3.2.2 shows that $\weight_{\omega}(x)= v_x(\divisor_{(\cX,\cX_{k,\redu})}(\omega)) +m$, so we obtain the required equality.

Otherwise, we consider the blow-up $h: \cX'^+ \rightarrow \cX^+$ at the closure of the center $\redu_{\cX}(x)$ of $x$ in $\cX^+$. By \cite{KollarMori}, Lemma 2.45\footnote{The proof in \cite{KollarMori} considers the case of a variety over a field, but it generalizes to varieties of finite type over a discrete valuation ring.} iterating this procedure a finite number of times, we obtain a log-regular model $\cY^+$ such that the $x$ corresponds to a component $E$ of the special fiber $\cY_k$. By reduction to the previous case, it is enough to check that the value $v_x(\divisor_{\cX^+}(\omega))$ does not change under such a blow-up $h$.

Let  $r=r_h +r_v$ be the codimension of $\redu_{\cX}(x)$ in $\cX$, where $r_h$ and $r_v$ are the number of irreducible components of $\overline{\lceil\Delta_X \rceil}$ and respectively of the special fiber $\cX_k$, containing $\redu_{\cX}(x)$. Let $E$ be the exceptional divisor of $h$. We denote the projections onto $S^+$ by $s_{\cX'}:\cX'^+\rightarrow S^+$ and $s_{\cX}:\cX^+ \rightarrow S^+$ and by $\pi$ a uniformizer in $R$. Then we have that
\begin{align*}
h^*(\omega^{\log}_{\cX'^+/S^+})& = h^* \big(\omega_{\cX'/R} \otimes \caO_{\cX'}(\overline{\lceil\Delta_X \rceil}_{\redu} + \cX'_{k,\redu} - s_{\cX'}^*\, \divisor(\pi)) \big) \\
& = \omega_{\cX/R}\otimes \caO_{\cX}((1-r)E +\overline{\lceil\Delta_X \rceil}_{\redu} + r_hE+ \cX_{k,\redu} + (r_v-1)E - s_\cX^*\,\divisor(\pi))= \omega^{\log}_{\cZ^+/S^+}.
\end{align*} It follows that $\divisor_{\cX'^+}(\omega)=h^*(\divisor_{\cX^+}(\omega))$ and this concludes the proof.
\end{proof}

\spa{We recall from \cite{MustataNicaise}, Section 4.7 that the Kontsevich-Soibelman skeleton $\Sk(X, \Delta_X, \omega)$ is the closure in $\text{Bir}(X)$ of the set $\text{Div}(X) $ of divisorial points of $X^\text{an}$ where the weight function $\weight_\omega$ reaches its minimal weight, namely $$\weight_{\omega}(X, \Delta_X) = \inf \{ \weight_{\omega}(x) \,| \, x \in \text{Div}(X) \} \in \R \cup \{ -\infty\}.$$

A priori the weight function associated to a rational pluricanonical form may have minimal weight $-\infty$, hence the corresponding Kontsevich-Soibelman skeleton would be empty. We prove that this does not occur for $\Delta_X$-logarithmic pluricanonical forms.
}

\begin{prop} \label{prop KS inside log-reg model}
Given a $\Delta_X$-logarithmic $m$-pluricanonical form $\omega$, for any log-regular model $\cX^+$ of $X^+$ the inclusion $\Sk(X,\Delta_X,\omega) \subseteq \Sk(\cX^+)$ holds.
\end{prop}
\begin{proof}
Let $\cX^+$ be a log-regular model of $X^+$ and let $y$ be a divisorial point of $X^\an$. It suffices to prove that $$\weight_{\omega}(y) \geqslant \weight_{\omega}(\rho_\cX(y))$$ and that the equality holds if and only if $y$ is in $\Sk(\cX^+)$. As in the proof of Lemma \ref{lemma log formula for weight function}, we consider the blow-up of $\cX^+$ at the closure of $\redu_{\cX}(y)$: iterating this procedure a finite number of times, we obtain a log-regular model $\cY^+$ such that $y \in \Sk(\cY^+)$.

Let $h: \cZ^+ \rightarrow \cW^+$ be a morphism of this sequence. If $\overline{ \{\redu_{\cW}(y)\}}$ is a stratum of $D_\cW$, then the morphism $h$ induces a subdivision of the skeleton $\Sk(\cW^+)$, so $\rho_\cZ(y)= \rho_{\cW}(y)$.

Otherwise, $\overline{ \{\redu_{\cW}(y)\}}$ is strictly contained in a stratum $V$ of $\D_\cW$. Let $E$ be the exceptional divisor of $h$, $r=r_h +r_v$ be the codimension of $V$ in $\cW$, where $r_h$ and $r_v$ are the number of irreducible components of $\overline{\lceil\Delta_X \rceil}$ and respectively of the special fiber $\cW_k$, containing $V$. Let $r+j$ be the codimension of $\overline{ \{\redu_{\cW}(y)\}}$, where $j \geqslant 1$. We denote the projections onto $S^+$ by $s_{\cW}:\cW^+\rightarrow S^+$ and $s_{\cZ}:\cZ^+ \rightarrow S^+$ and by $\pi$ a uniformizer in $R$. Then we have that
\begin{align*}
h^*&(\omega^{\log}_{\cW^+/S^+}) = h^* \big(\omega_{\cW/R} \otimes \caO_{\cW}(\overline{\lceil\Delta_X \rceil}_{,\redu} + \cW_{k,\redu} - s_\cW^*\, \divisor(\pi)) \big) \\
& = \omega_{\cZ/R}\otimes \caO_{\cZ}((1-r-j)E +\overline{\lceil\Delta_X \rceil}_{,\redu} + r_hE+ \cZ_{k,\redu} + (r_v-1)E - s_\cZ^*\,\divisor(\pi))\\
& = \omega^{\log}_{\cZ^+/S^+} \otimes \caO_{\cZ}(-j E).
\end{align*} It follows that $\divisor_{\cZ^+}(\omega)= h^*(\divisor_{\cW^+}(\omega)) + mjE$, so
\begin{align*}
v_{\rho_\cZ(y)}(\divisor_{\cZ^+}(\omega))& = v_{\rho_\cZ(y)}\big(h^*(\divisor_{\cW^+}(\omega)) + mjE\big)  \\ & \geqslant v_{\rho_\cW(y)}(\divisor_{\cW^+}(\omega)) + mj v_{\rho_\cW(y)}(E) \\ & > v_{\rho_\cW(y)}(\divisor_{\cW^+}(\omega))
\end{align*} where for the first inequality we apply \cite{MustataNicaise}, Proposition 3.1.6, while the second strict inequality holds as $j >0$ and  $v_{\rho_\cW(y)}(E) >0 $ since the center of the valuation $\rho_\cW(y)$ is contained in $E$. Therefore, for any such morphism $h$, the weight is strictly increasing, namely $\weight_{\omega}(\rho_\cZ(y)) > \weight_{\omega} (\rho_\cW(y))$. This concludes the proof.
\end{proof}

Only the components of $\Delta_X$ with coefficient $a_i=1$ determine strata that are contained in the Kontsevich-Soibelman skeletons. The introduction of $\Delta_X$-log pluricanonical forms allows us to construct non-empty Kontsevich-Soibelman skeletons even for varieties with Kodaira dimension $-\infty$, as in the following examples:

\begin{example}
Let $X$ be the projective line $\mathbb{P}^1_K$ with affine coordinates $x$ and $y$, and $\Delta_X=(0:1) + (1:0)$. Then $a_i=1$ for any $i$ and there exist $\Delta_X$-logarithmic canonical forms. For example, we consider $$\omega =  \frac{dx}{x}= -\frac{dy}{y}.$$ Let $\cX= \mathbb{P}_R^1$ and $D_\cX= (0:1) + (1:0)+ \mathbb{P}^1_k$. The log scheme $\cX^+=(\cX,D_\cX)$ is a log-regular model of $X^+=(X,\lceil \Delta_X \rceil)$ and the associated skeleton $\Sk(\cX^+)$ looks like this:
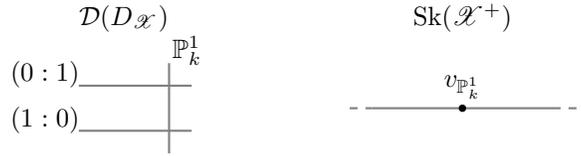
\begin{center}
\begin{tikzpicture}[scale=0.3]
\draw[gray, thick] (-1,1) -- (-1,-3);
\draw[gray, thick] (-5,0) -- (0,0);
\draw[gray, thick] (-5,-2) -- (0,-2);
\node[above] at (-6.5,-0.5) {$(0:1)$};
\node[above] at (-6.5,-2.5) {$(1:0)$};
\node[above] at (-0.2,0.5) {$\mathbb{P}{^1_k}$};
\node[above] at (-3,2) {$\mathcal{D}(D_\cX)$};
\draw[gray, thick] (8,-1) -- (12,-1) -- (16,-1);
\draw [gray, thick, dashed] (7,-1) -- (8,-1);
\draw [gray, thick, dashed] (16,-1) -- (17.5,-1);
\filldraw[black] (12,-1) circle (4pt) node[above] {$v_{\mathbb{P}^1_k}$};
\node[above] at (12,2) {$\Sk(\cX^+)$};
\end{tikzpicture}
\end{center}
Since $\divisor_{\cX^+}(\omega)= 0$, the weight associated to $\omega$ is minimal at any point of the skeleton $\Sk(\cX^+)$. Thus $\Sk(X,\Delta_X,\omega)=\Sk(\cX^+) \simeq \R$.
\end{example}
\begin{example}
Let $X=\mathbb{P}^1_K$ and $\Delta_X=\frac{2}{3}(0:1) + \frac{2}{3}(1:0)+\frac{2}{3}(1:1)$. So $a_i=\frac{2}{3}$ for any $i$ and there exist $\Delta_X$-logarithmic $3$-pluricanonical forms. We set $$\omega =  \frac{1}{(x-1)^2} \cdot \frac{1}{x^2} (dx)^3= -\frac{1}{(1-y)^2}\cdot \frac{1}{y^2} (dy)^3.$$ We consider $\cX= \mathbb{P}_R^1$ and $D_\cX= (0:1) + (1:0)+ (1:1)+ \mathbb{P}^1_k$, then $\cX^+=(\cX,D_\cX)$ is a log-regular model of $X^+=(X,\lceil \Delta_X \rceil)$ and $\Sk(\cX^+)$ is 
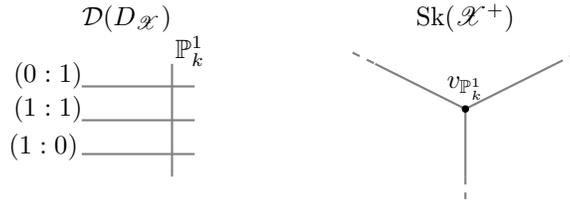
\begin{center}
\begin{tikzpicture}[scale=0.3]
\draw[gray, thick] (-1,1) -- (-1,-4);
\draw[gray, thick] (-5,0) -- (0,0);
\draw[gray, thick] (-5,-1.5) -- (0,-1.5);
\draw[gray, thick] (-5,-3) -- (0,-3);
\node[above] at (-6.5,-0.5) {$(0:1)$};
\node[above] at (-6.5,-2) {$(1:1)$};
\node[above] at (-6.7,-3.7) {$(1:0)$};
\node[above] at (-0.2,0.5) {$\mathbb{P}{^1_k}$};
\node[above] at (-3,2) {$\mathcal{D}(D_\cX)$};
\draw[gray, thick] (8,1) -- (12,-1);
\draw[gray, thick] (16,1) -- (12,-1);
\draw[gray, thick] (12,-4) -- (12,-1);
\draw [gray, thick, dashed] (7,1.5) -- (8,1);
\draw [gray, thick, dashed] (16,1) -- (17,1.5);
\draw [gray, thick, dashed] (12,-4) -- (12,-5);
\filldraw[black] (12,-1) circle (4pt) node[above] {$v_{\mathbb{P}^1_k}$};
\node[above] at (12,2) {$\Sk(\cX^+)$};
\end{tikzpicture}
\end{center}
Since $\divisor_{\cX^+}(\omega)= (0:1) + (1:0)+ (1:1)$, the weight associated to $\omega$ is minimal at the divisorial point $v_{\mathbb{P}_k^1}$ corresponding to $\mathbb{P}^1_k$ and is strictly increasing with slope $1$ along the unbounded edges, when we move away from the point $v_{\mathbb{P}^1_k}$. Therefore, $\Sk(X,\Delta_X,\omega)=\{v_{\mathbb{P}^1_k}\}$.
\end{example}

\subsection{Weight function and Kontsevich-Soibelman skeleton for products.} \label{paragraph weight function product}
\spa{Let $\cX^+$ and $\cY^+$ be log-smooth models over $S^+$ of $X^+=(X,\lceil\Delta_X\rceil)$ and $Y^+=(Y,\lceil\Delta_Y\rceil)$ respectively. Then, the $fs$ fibred product $\cZ^+=\cX^+  \times^{\text{fs}}_{S^+} \cY^+$ is a log-regular model of $Z^+:=X^+ \times^{\text{fs}}_K Y^+$. Therefore, given $\omega_{X^+}$ and $\omega_{Y^+}$ $\Delta_X$-logarithmic and $\Delta_Y$-logarithmic $m$-pluricanonical forms on $(X,\Delta_X)$ and $(Y,\Delta_Y)$ respectively, the form $$\varpi=\pr_{X^+}^* \,\omega_{X^+} \otimes \pr_{Y^+}^* \,\omega_{Y^+}$$ is a $\Delta_Z$-logarithmic $m$-pluricanonical form on $(Z,\Delta_Z)$, where $\Delta_Z= X \times_K \Delta_Y + \Delta_X \times_K Y$. Viewing these forms as rational sections of logarithmic $m$-pluricanonical bundles, we see that $\divisor_{\cZ^+}(\varpi)=\pr_{\cX^+}^* \big(\divisor_{\cX^+}(  \omega_{X^+})\big) +  \pr_{\cY^+}^* \big(\divisor_{\cY^+} (\omega_{Y^+})\big)$ according to (\ref{equ log can bundles}).}

\spa{Let $z$ be a point of $F_{\cZ} \cap \cZ_k$; as before, we denote by $x$ and $y$ the images of $z$ under the local isomorphism $F_{\cZ} \rightarrow F_{\cX} \times_{F_S} F_{\cY}$. Any morphism $\varepsilon \in \sigma_z$ defines a point $v_{z,\varepsilon}$ in $\Sk(\cZ^+)$. For the sake of convenience, we simply denote the valuations by the corresponding morphism and we denote $\alpha= \pr_{\cX}(\varepsilon)$ and $\beta=\pr_{\cY}(\varepsilon)$. We aim to relate the valuation $v_{\varepsilon}\big(\divisor_{\cZ^+}(\varpi)\big)$ to the values $$v_{\alpha}\big(\divisor_{\cX^+}(\omega_{X^+}\big)) \text{ , } v_{\beta}\big(\divisor_{\cY^+}(\omega_{Y^+})\big).$$}

\spa{Let $f_x \in \caO_{\cX,x}$ be a local equation of $\divisor_{\cX^+}(\omega_{X^+})$ around $x$. In order to evaluate $v_{x,\alpha}$ on $f_x$, we consider an admissible expansion of $f_x$ as in (\ref{eq-adm}) $$f_x=\sum_{\gamma\in \caC_{\cX,x}}c_\gamma \gamma.$$ Furthermore, this expansion induces also an expansion of $\pr_{\cX}^*(f_x)$ by  $$\pr_{\cX}^*(f_x)=\sum_{\gamma\in \caC_{\cX,x} }\pr^*_{\cX}(c_\gamma) \gamma$$ as formal power series in $\widehat{\caO}_{\cZ,z}$, since the morphism of characteristic sheaves $\mathcal{C}_{\cX,x} \hookrightarrow \mathcal{C}_{\cZ,z}$ is injective. Following the same procedure for a local equation $f_y \in \caO_{\cY,y}$ of $\divisor_{\cY^+}(\omega_{Y^+})$ around $y$, we get an expansion of $f_y$ that extends to $\pr_{\cY}^*(f_y)$: 
$$f_y=\sum_{\delta\in \caC_{\cY,y}}d_\delta \delta.$$}

\spa{A local equation of $\varpi$ around $z$ is determined by $\pr_{\cX}^*(f_x)\,\,\pr_{\cY}^*(f_y)$. Thus $$v_{\varepsilon}(\divisor_{\cZ^+}(\varpi))
= v_{\varepsilon}\big(\pr_{\cX}^*(f_x) \,\,\pr_{\cY}^*(f_y)\big)$$ and by multiplicativity of the valuation $v_{\varepsilon}$ $$v_{\varepsilon}\big(\pr_{\cX}^*(f_x) \,\,\pr_{\cY}^*(f_y)\big) = v_{\varepsilon}\big(\pr_{\cX}^*(f_x)\big) + v_{\varepsilon}\big(\pr_{\cY}^*(f_y)\big).$$ Recalling Remark \ref{rem def-val}, the valuation can be computed as follows $$ v_{\varepsilon}\big(\pr_{\cX}^*(f_x)\big) = \min\{\varepsilon(\gamma)\,|\,c_\gamma \neq 0\};$$ as the elements $\gamma$ belong to $\caC_{\cX,x}$ and $\alpha$ is defined to be $\pr_{\cX}(\varepsilon)$, we have $$\min\{\varepsilon(\gamma)\,|\,c_\gamma \neq 0\} = \min\{\alpha(\gamma)\,|\,c_\gamma \neq 0\} =  v_{x,\alpha}(f_x).$$ Hence, we conclude that \begin{align} \label{equ valuations product}
\begin{split}
v_{\varepsilon}\big(\divisor_{\cZ^+}(\varpi)\big)
& = v_{\varepsilon}\big(\pr_{\cX}^*(f_x)\big) + v_{\varepsilon}\big(\pr_{\cY}^*(f_y)\big)\\
& = v_{\alpha}(f_x) + v_{\beta}(f_y) \\
& = v_{\alpha}\big(\divisor_{\cX^+}(\omega_{X^+})\big) + v_{\beta}\big(\divisor_{\cY^+}(\omega_{Y^+})\big).
\end{split}
\end{align}
}

\spa{This result turns out to be advantageous to compute the weight function $\weight_\varpi$ on divisorial points of $\Sk(\cZ^+)$: 
\begin{align} \label{equ weight function product}
\begin{split}
\weight_\varpi(\varepsilon)
& = v_{\varepsilon}\big(\divisor_{\cZ^+}(\varpi)\big) +m \\
& = v_{\alpha}\big(\divisor_{\cX^+}(\omega_{X^+})\big) + v_{\beta}\big(\divisor_{\cY^+}(\omega_{Y^+})\big) + m \\
& = \weight_{\omega_{X^+}}(\alpha) + \weight_{\omega_{Y^+}}(\beta) - m.
\end{split}
\end{align}
}

\spa{Under the notations of the previous paragraphs, our computations lead to the following result.
\begin{thm}  \label{thm semistability and KS skeletons}
Suppose that the residue field $k$ is algebraically closed and that $\cX^+$ is semistable. Then, the PL homeomorphism of skeletons $$\Sk(\cZ^+) \xrightarrow{\sim} \Sk(\cX^+) \times \Sk(\cY^+)$$ given in Proposition \ref{prop semistability and skeletons} restricts to a PL homeomorphism of Kontsevich-Soibelman skeletons $$\Sk(Z,\Delta_Z, \varpi) \xrightarrow{\sim} \Sk(X,\Delta_X, \omega_{X^+}) \times \Sk(Y,\Delta_Y, \omega_{Y^+}).$$ 
\end{thm}
\begin{proof}
This follows immediately from the equality (\ref{equ weight function product}) that shows that a point in $\Sk(Z,\Delta_Z)$ has minimal value $\weight_{\varpi}(Z,\Delta_Z)$ if and only if its projections have minimal value $\weight_{\omega_{X^+}}(X,\Delta_X)$ and $\weight_{\omega_{Y^+}}(Y,\Delta_Y)$.
\end{proof}
}

\section{The essential skeleton of a product}\label{sec birational}
We will need a few notions from birational geometry, see \cite{KollarMori}. Let $(X,\Delta_X)$ be a pair over $K$ such that $X$ is normal, $\Delta_X$ is an effective divisor, and $K_X+\Delta_X$ is $\mathbb{Q}$-Cartier. Then we say that $(X,\Delta_X)$ is \emph{log canonical} if for every log resolution $f \colon Z \to X$, in the formula

\[
K_Z+\Delta_Z = f^*(K_X+\Delta_X)+\sum a_D D,
\]
where $\Delta_Z$ is the strict transform of $\Delta_X$ plus the reduced exceptional divisor, all the $a_D$ are non-negative. The sum ranges over all components of $\Delta_Z$. In fact the quantity $a_D$, called the \emph{log discrepancy} of $D$ with respect to $(X,\Delta_X)$, depends only on the valuation corresponding to $D$, and this condition needs only be tested on a single log resolution.

We will be most interested in the case where $(X,\Delta_X)$ satisfies the stronger condition of being \emph{divisorially log terminal}, or dlt. A closed subset $Y \subset X$ is called a \emph{log canonical center} if for some (respectively any) log resolution, $Y$ is the image of a divisor $D$ with $a_D=0$. The pair $(X,\Delta_X)$ is said to be dlt if for every log canonical center $Y$, there is a neighborhood of the generic point of $Y$ where $(X,\Delta_X)$ is snc.

We say that a pair $(X,\Delta)$ is \emph{Kawamata log terminal}, or klt, if it is dlt and the coefficients of $\Delta$ are all strictly less than $1$.

\subsection{Essential skeleton of a pair.}
\spa{Let $(X,\Delta_X)$ be a pair such that $X^+=(X,\lceil \Delta_X \rceil)$ is a log-regular log scheme. Let $\omega$ be a non-zero regular $\Delta_X$-logarithmic $m$-pluricanonical form on $(X,\Delta_X)$. Let $\cX^+$ be a log-smooth model of $X^+$ over $S^+$. There exist minimal positive integers $d$ and $n$ such that the divisor $\divisor_{\cX^+}(\omega^{\otimes d} \pi^{-n})$ is effective and the multiplicity of some component of the special fiber is zero: we denote this divisor by $D_{\min}(\cX,\omega)$. It follows from the properties of the weight function (see \cite{MustataNicaise}, Proposition 4.5.5) that for any $x \in \Sk(\cX^+)$ $$v_x(D_{\min}(\cX,\omega)) + dm = \weight_{\omega^{\otimes d}\pi^{-n}}(x)= d \cdot \weight_{\omega}(x) + v_x(\pi^{-n}) = d \cdot \weight_{\omega}(x) -n.$$
}
\begin{lemma} \label{lemma characterization minimal points}
Let $v_{x,\alpha}$ be a divisorial point in $\Sk(\cX^+)$, then $v_{x,\alpha} \in \Sk(X,\Delta_X, \omega)$ if and only if $D_{\min}(\cX,\omega)$ does not contain $x$.
\end{lemma}
\begin{proof}
We denote $v_{x,\alpha}$ simply by $\alpha$. By the above series of equalities, the weight function $\weight_{\omega}$ reaches its minimum at $\alpha$ if and only if $v_\alpha(D_{\min}(\cX,\omega))$ is minimal, hence in particular equal to zero.

Let $h: \cY^+ \rightarrow \cX^+$ be a sequence of blow-up morphisms of strata of $D_\cX$ such that $\alpha$ corresponds to an irreducible component $E$ of $D_\cY$. As in the proof of Proposition \ref{prop KS inside log-reg model} $$h^*(D_{\min}(\cX,\omega))=h^*(\divisor_{\cX^+} (\omega^{\otimes d} \pi^{-n}))= \divisor_{\cY^+} (\omega^{\otimes d} \pi^{-n}).$$ Therefore we have that $v_\alpha(D_{\min}(\cX,\omega)) >0$ if and only if $E \subseteq \divisor_{\cY^+} (\omega^{\otimes d} \pi^{-n})$, and this holds if and only if $x \in D_{\min}(\cX,\omega)$.
\end{proof}

\spa{We define the essential skeleton $\Sk(X,\Delta_X)$ of $(X,\Delta_X)$ as the union of all Kontsevich-Soibelman skeletons $\Sk(X, \Delta_X,\omega)$, where $\omega$ ranges over all regular $\Delta_X$-logarithmic pluricanonical forms. In the case of an empty boundary, this recovers the notions introduced in \cite{MustataNicaise}.

The further reason to define the essential skeleton this way is that it behaves nicely under birational morphisms. Let $f \colon X'\to X$ be a log resolution. Then there is a $\mathbb{Q}$-divisor $\Gamma'$ with snc support, and no coefficient exceeding $1$, such that $K_{X'}+\Gamma' = f^*(K_X+\Delta_X)$. Take $\Delta_{X'}$ to be the positive part of $\Gamma'$ and write $$K_{X'}+\Delta_{X'}= f^*(K_{X}+\Delta_X)+N.$$ For any $m$, pullback along with multiplication by the divisor of discrepancies $N$ induces an isomorphism of vector spaces \begin{equation} \label{equ isomo forms log resolution}
H^0(X, mK_{X}+m\Delta_X) \cong H^0(X', mK_{X'}+m\Delta_{X'}).
\end{equation} Let $\omega$ and $\omega'$ be corresponding forms via this isomorphism.}
\begin{prop} \label{prop birational invariance essential skeleton for pairs}
Under the identification of the birational points of $X$ with those of $X'$, $\Sk(X, \Delta_X,\omega)$ is identified with $\Sk(X', \Delta_{X'},\omega')$.
\end{prop}
\begin{proof}
The Kontsevich-Soibelman skeleton $\Sk(X,\Delta_X,\omega)$ is contained in the skeleton associated to any log-regular model by Proposition \ref{prop KS inside log-reg model}, so we choose log-regular models $\cX^+$ and $\cX'^+$ so that $f$ extends to a log resolution $f_R \colon \cX' \to \cX$. We denote by $\Delta_\cX = \overline{\Delta_X} + \cX_{k,\redu}$ and $\Delta_{\cX'}=\overline{\Delta_{X'}} + \cX'_{k,\redu}$. It suffices to check the proposition for divisorial valuations. By Lemma \ref{lemma characterization minimal points} a divisorial valuation $v$ of $\Sk(\cX^+)$ is in $\Sk(X,\Delta_X,\omega)$ if and only if it is a log canonical center of $(\cX,\Delta_{\cX})$ and the divisor $D_{\min}(\cX,\omega)$ does not contain the center of $v$.  

Suppose $v$ is a divisorial point in $\Sk(X, \Delta_X,\omega)$. Without loss of generality, we can assume that the divisor $\divisor_{\cX^+}(\omega)$ in $\cX^+$ does not contain the center of $v$. As $(\cX, \Delta_\cX)$ is a dlt pair, we have that $$K_{\cX'}+\Delta_{\cX'}= f_R^*(K_{\cX}+\Delta_\cX)+M$$ where $M$ is effective, thus  $\divisor_{\cX'^+}(\omega')=f^*_R(\divisor_{\cX^+}(\omega))+mM$. As $v$ is a log canonical center of $(\cX,\Delta_\cX)$, $M$ does not vanish along $v$, so neither does $\divisor_{\cX'^+}(\omega')$. Likewise $v$ is a log canonical center of $(\cX', \Delta_{\cX'})$. It follows that $v \in \Sk(X',\Delta_{X'},\omega')$.

Conversely, if $v$ is a divisorial point in $\Sk(X', \Delta_{X'},\omega')$, it is a log canonical center of $(\cX', \Delta_{\cX'})$ and the divisor $D_{\min}(\cX',\omega')$ does not contain the center of $v$. As a result, $v$ is also a log canonical center of $(\cX, \Delta_\cX)$, and the divisor $D_{\min}(\cX,\omega)$ does not contain the center of $v$ since its pullback does not.
\end{proof}
We define the Kontsevich-Soibelman skeleton $\Sk(X,\Delta_X,\omega)$ of a dlt pair $(X, \Delta_X)$ as the Kontsevich-Soibelman skeleton $\Sk(X',\Delta_{X'},\omega')$ where $(X',\Delta_{X'})$ is any log-resolution of $(X,\Delta_X)$, and $\omega'$ is the forms corresponding to $\omega$ under the isomorphism (\ref{equ isomo forms log resolution}): Proposition \ref{prop birational invariance essential skeleton for pairs} guarantees that this is well-defined. It follows that we can define the essential skeleton $\Sk(X,\Delta_X)$ of a dlt pair $(X, \Delta_X)$ as the essential skeleton of any log-resolution of $(X,\Delta_X)$.

Moreover, notice that our construction works more generally for log canonical pairs, hence the notions of Kontsevich-Soibelman skeletons and essential skeleton generalize to such pairs.

\spa{Suppose that $(X,\Delta_X)$ is a dlt pair over $K$, such that $K_X + \Delta_X$ is semiample. Suppose thar $\cX$ is a good dlt minimal model of $(X,\Delta_X)$ over $R$ and let $\Delta_\cX=\overline{\Delta_X} + \cX_{k,\redu}$. We denote by $\mathcal{D}_0^{=1}(\cX, \Delta_\cX)$ the dual complex of the strata of the coefficient 1 part of $\Delta_{\cX}$ that lie in the special fier.

We consider a log resolution $f:X' \rightarrow X$ that extends to a log resolution of $(\cX,\Delta_\cX)$. We write $K_{X'} + \Gamma' = f^*(K_X + \Delta_X)$. Let $\Delta_{X'}$ be the positive part of $\Gamma'$. We may embed the open dual complex $\mathcal{D}_0^{=1}(\cX, \Delta_\cX)$ into the birational points of $X$.
}
\begin{prop} \label{prop dual complex minimal good dlt = essential skeleton any resolution}
This embedding identifies $\mathcal{D}_0^{=1}(\cX, \Delta_\cX)$ with $\Sk(X',\Delta_{X'})$.
\end{prop}
\begin{proof}
Choose a regular $R$-model $\cX'$ for $X'$ which is a log resolution of $(\cX, \Delta_\cX)$ extending $f$ and let  $\Delta_{\cX'}=\overline{\Delta_{X'}} + \cX'_{k,\redu}$, namely we have the log resolutions \begin{align*}
&(X', \Delta_{X'}) \rightarrow (X,\Delta_X) & \text{where }(X,\Delta_X) \text{ dlt and }\Delta_{X'} \text{ is the positive part of }\Gamma'\\
&(\cX', \Delta_{\cX'}) \rightarrow (\cX,\Delta_\cX) & \text{where }\Delta_\cX=\overline{\Delta_X} + \cX_{k,\redu} \text{ and } \Delta_{\cX'}=\overline{\Delta_{X'}} + \cX'_{k,\redu}.
\end{align*} As in the previous proof, it suffices to check the proposition for divisorial valuations. Let $v$ be a divisorial valuation, and suppose $v \in \mathcal{D}_0^{=1}(\cX,\Delta_\cX)$. Then $v$ is a log canonical center for $(\cX,\Delta_\cX)$, so $v$ is also a log canonical center of $(\cX', \Delta_{\cX'})$. For a sufficiently divisible index, we may find a $\Delta_X$-logarithmic pluricanonical form on $(X,\Delta_X)$ whose associated divisor in $\cX$ has vanishing locus $C$ such that $C$ is a divisor not containing the center of $v$. After pullback, we get a $\Delta_{X'}$-logarithmic pluricanonical form $\omega'$ whose associated divisor in $\cX'$ is supported on the strict transform of $C$ and the exceptional divisors of positive log discrepancy. But none of these contain $v$. Thus, $v \in \Sk(X',\Delta_{X'},\omega')$.

Conversely, if $v$ is a divisorial point in $\Sk(X', \Delta_{X'})$, then $v$ is a log canonical center of $(\cX', \Delta_{\cX'})$, so $v$ is a log canonical center of $(\cX, \Delta_\cX)$, hence an element of the open dual complex $\mathcal{D}_0^{=1}(\cX, \Delta_\cX)$.
\end{proof}

\begin{rem}
Proposition \ref{prop dual complex minimal good dlt = essential skeleton any resolution} compares the essential skeleton of $(X,\Delta_X)$ to the skeleton of a good minimal dlt model of $(X,\Delta_X)$. Thus, the result can be restated as follows: if $(X,\Delta_X)$ is a dlt pair with $K_X+\Delta_X$ semiample and $(\cX,\Delta_\cX)$ is a good dlt minimal model of $(X,\Delta_X)$ over $R$, then $\mathcal{D}_0^{=1}(\cX, \Delta_\cX) = \Sk(X,\Delta_X)$. This generalizes \cite{NicaiseXu}, Theorem 3.3.3 to dlt pairs.
\end{rem}

\subsection{Essential skeletons and products of log-regular models.}

\spa{We say that a pair $(X,\Delta_X)$ has non-negative Kodaira-Iitaka dimension if some multiple of the line bundle $K_X + \Delta_X$ has a regular section.}
\begin{theorem} \label{thm essential skeleton of product via log-reg}
Assume that the residue field $k$ is algebraically closed. Let $(X,\Delta_X)$ and $(Y,\Delta_Y)$ be pairs such that $X^+=(X,\lceil \Delta_X \rceil)$ and $Y^+=(Y, \lceil \Delta_Y\rceil)$ are log-regular log scheme over $K$. Suppose that both pairs have non-negative Kodaira-Iitaka dimension and both admit semistable log-regular models $\cX^+$ and $\cY^+$ over $S^+$. Then, the PL homeomorphism of skeletons $$\Sk(\cZ^+) \xrightarrow{\sim} \Sk(\cX^+) \times \Sk(\cY^+)$$of Proposition \ref{prop semistability and skeletons} induces a PL homeomorphism of essential skeletons $$\Sk(Z,\Delta_Z) \xrightarrow{\sim} \Sk(X,\Delta_X) \times \Sk(Y,\Delta_Y)$$ where $\cZ^+$, $Z$ and $\Delta_Z$ are the respective products. 
\end{theorem}
\begin{proof}
It follows immediately from Theorem \ref{thm semistability and KS skeletons} that we have the inclusion $ \Sk(X,\Delta_X) \times \Sk(Y,\Delta_Y) \subseteq \Sk(Z,\Delta_Z)$. Thus, we reduce to prove the following statement. Let $v_{z,\varepsilon}$ be a divisorial point in $\Sk(\cZ^+)$ and $(v_{x,\alpha},v_{y,\beta})$ be the corresponding pair in $\Sk(\cX^+) \times \Sk(\cY^+)$ under the isomorphism of Proposition \ref{prop semistability and skeletons}; if $v_{z,\varepsilon}$ lies in the essential skeleton $\Sk(Z,\Delta_Z)$, then $v_{x,\alpha}$ lies in $\Sk(X,\Delta_X)$.

Assume that $v_{z,\varepsilon}$ lies in the essential skeleton $\Sk(Z,\Delta_Z)$. Then there exists a non-zero regular $\Delta_Z$-logarithmic $m$-pluricanonical form $\omega$ on $Z^+$, such that $v_{z,\varepsilon} \in \Sk(Z,\Delta_Z, \omega)$. By Lemma \ref{lemma characterization minimal points}, $D_{\min}(\cZ,\omega)$ does not contain $z$.

Let $E$ be an irreducible component of $\cY_k$ containing $y$ and denote by $\xi_E$ the generic point of $E$. Then the point in the Kato fan of $\cZ^+$ corresponding to $(x,\xi_E)$ is not contained in $D_{\min}(\cZ,\omega)$, as otherwise $z$ would be contained in it.

As $k$ is algebraically closed, we can choose a $k$-rational point $p$ in $E$ such that $p$ is contained in no other components of $D_\cY$ and $D_{\min}(\cZ,\omega)$ does not contained the locus $\overline{\{x\}} \times_R \{p\}$. By Hensel's Lemma and the assumption of semistability, $p$ can be lifted to an $R$-rational point of $\cY$. The pullback of $\cZ^+$ along this $R$-rational point is an embedding $i: \cX^+ \rightarrow \cZ^+$, so we have the diagram
\[
\xymatrix{ \cX^+ \ar[r]^i\ar[d] & \cZ^+ \ar[r]^{\pr_\cX}\ar[d]^{\pr_\cY} & \cX^+ \ar[d] \\
S \ar[r] & \cY^+\ar[r] & S.
}
\] Since $S$ has trivial normal bundle in $\cY$, we  have that $$\omega_{X^+/K}^{\log} = i^*\big(\omega_{Z^+/K}^{\log}\big),$$ so $i^*(\omega)$ is a non-zero pluricanonical form on $X$ and in particular is a regular $\Delta_X$-logarithmic $m$-pluricanonical form. Moreover, $D_{\min}(\cX,i^*(\omega)) = i^*(D_{\min}(\cZ,\omega))$. Finally, $x$ is not contained in $D_{\min}(\cX, i^*(\omega))$, as otherwise $i(x)= \{x\} \times_R \{p\}$ would be contained in $D_{\min}(\cZ,\omega)$. By Lemma \ref{lemma characterization minimal points}, $x$ is a point of $\Sk(X,\Delta_X, i^*(\omega))$ and this concludes the proof.
\end{proof}
\begin{rem}
Consider the case where the line bundles $K_X + \Delta_X$ and $K_Y + \Delta_Y$ are semi-ample, i.e. some multiple of them is base point free. It follows from the arguments of \cite{NicaiseXu}, Theorem 3.3.3 that the essential skeleton of $(Z,\Delta_Z)$ is a finite union of Kontsevich-Soibelman skeletons where the union runs through a generating set of global sections of a sufficiently large multiple of $K_Z+\Delta_Z$. We can construct such a set from generating sets of global sections of multiples of $K_X + \Delta_X$ and $K_Y+ \Delta_Y$ respectively, via tensor product. Then, in this case, the result of Theorem \ref{thm essential skeleton of product via log-reg} follows directly from Theorem \ref{thm semistability and KS skeletons}.
\end{rem}

\subsection{Essential skeletons and products of dlt models.} 
\spa{We need a combinatorial lemma to understand the formally local behavior of products of semistable dlt models.}
\begin{lemma}\label{smalltoric}
Let $M$ be the monoid generated by $r_1 \ldots r_{n_1}$, $s_1 \ldots s_{n_2}$ with the single relation $\sum_{i=1}^{n_1} r_i = \sum_{j=1}^{n_2} s_j$. Then any small $\mathbb{Q}$-factorialization of the affine toric variety $W=\Spec(k[M])$ associated to $M$ is a log resolution.
\end{lemma}

\begin{proof}
We calculate the fan of $W$. Let $N$ be the dual lattice of $M$. The fan associated to $W$ is the cone of elements of $N\otimes \mathbb{R}$ which are non-negative on $M$. We consider these as linear functions $l$ on the vector space spanned by the $r_i$ and $s_i$, subject to the restriction that $l(\sum_{i=1}^{n_1} r_i)=l( \sum_{j=1}^{n_2} s_j)$. Let $x_{ij}$ be the function which is $1$ on $r_i$ and $s_j$ and $0$ on all others. Then the fan of $W$ is given by the single cone $C_W$ spanned by the $x_{ij}$.

Any $\mathbb{Q}$-factorialization $\widetilde{W}$ corresponds to a simplicial subdivision of the cone $C_W$ (see \cite{Fulton1993}, p.65). We now check that every choice of $\widetilde{W}$ is non-singular.

A maximal cone of $\widetilde{W}$ is spanned by $n=n_1+n_2-1$ independent rays of $C_W$. Each ray of $C_W$ corresponds to a choice of $x_{ij}$, and we can index these by edges of the complete bipartite graph $B$ on the $r_i$ and $s_j$. These $x_{ij}$ are independent if and only if the corresponding edges form a spanning tree. Let $w_1 \ldots w_{n}$ span a maximal cone of $\widetilde{W}$. On this affine chart, $\widetilde{W}$ is smooth if and only if the $w_i$ generate $N$ as a lattice. We have shown already that the $x_{ij}$ generate $N$. But every $x_{ij}$ is either one of the $w_i$, or it completes a cycle in $B$, so that it is a $\mathbb{Z}$-linear combination of the $w_i$.
\end{proof}

\begin{prop} \label{prop product of dlt minimal pairs}
Let $(\cX,\Delta_{\cX})$ and $(\cY,\Delta_{\cY})$ be semistable good projective dlt minimal pairs over the germ of a pointed curve $\cC$. The product $(\cZ,\Delta_{\cZ})$ is a log canonical pair, $K_\cZ+\Delta_{\cZ}$ is semiample, and the log canonical centers of $(\cZ,\Delta_{\cZ})$ are strata of the coefficient $1$ part of $\Delta_{\cZ}$.
\end{prop}

\begin{proof}
The product $\cZ$ is normal as $\cX$ and $\cY$ are semistable. The divisor $K_\cZ+\Delta_{\cZ}$ is semiample by pullback of semiample divisors.

Let $\widetilde{\cX}$ and $\widetilde{\cY}$ be log resolutions of $\cX$ and $\cY$. Then, we have 
\begin{align*}
K_{\widetilde{\cX}}+ \Delta_{\widetilde{\cX}}&= f_\cX^*(K_\cX + \Delta_{\cX}) + \sum a_i E_{\cX,i} \\
K_{\widetilde{\cY}}+ \Delta_{\widetilde{\cY}}&= f_\cY^*(K_\cY + \Delta_{\cY}) + \sum b_j E_{\cY,j}
\end{align*} where the coefficients $a_i$ and $b_j$ are non-negative. Let $\widetilde{\cZ}$ be a toroidal log resolution of the fs product $\widetilde{\cX} \times^{\text{fs}}_\cC \widetilde{\cY}$. In particular, $\widetilde{\cZ}$ is a log resolution of $\cZ$ and we can write $$K_{\widetilde{\cZ}}+ \Delta_{\widetilde{\cZ}}= f_\cZ^*(K_\cZ + \Delta_{\cZ}) + \sum c_h E_{\cZ,h}$$ where $\Delta_{\widetilde{\cZ}}$ is effective. Over the generic fiber, $(\cZ,\Delta_{\cZ})$ is dlt, so we need only compute discrepancies over the special fiber, namely study the positivity of the coefficients $c_h$. 

Let $\Gamma$ be a divisor of $\widetilde{\cZ}$ over the special fiber, denote by $v_\Gamma$ the corresponding divisorial valuation in $\Sk(\widetilde{\cZ}^+)$, by $\Gamma_\cX$, $\Gamma_\cY$ and $\Gamma_{\cZ}$ its images in $\cX$, $\cY$ and $\cZ$. The projections of $v_\Gamma$ in $\Sk(\widetilde{\cX}^+)$ and $\Sk(\widetilde{\cY}^+)$ are divisorial valuations. Up to subdivisions of the skeletons, we can assume without loss of generality that the projections correspond to divisors $\Gamma_{\widetilde{\cX}}$ and $\Gamma_{\widetilde{\cY}}$.

Choose a $\Delta_{\cX}$-logarithmic and a $\Delta_{\cY}$-logarithmic pluricanonical forms $\omega_{\cX}$ on $\cX$ and  $\omega_{\cY}$ on $\cY$ respectively, such that the divisors $\divisor_{\cX^+}(\omega_\cX)$ and $\divisor_{\cY^+}(\omega_{\cY})$ do not contain $\Gamma_\cX$ and $\Gamma_\cY$ respectively, where $\cX^+=(\cX,\lceil \Delta_{\cX}\rceil)$ and  $\cY^+=(\cY,\lceil \Delta_{\cY}\rceil)$. Then, the divisor $\divisor_{\cZ^+}(\omega_{\cZ})$, associated to the wedge product $\omega_{\cZ}$ of the pullbacks $\omega_{\cX}$ and $\omega_{\cY}$ to $\cZ$, does not contain $\Gamma_{\cZ}$.  Denote by $\omega_{\widetilde{\cX}}$, $\omega_{\widetilde{\cY}}$ and $\omega_{\widetilde{\cZ}}$ the pullback of the respective forms to $\widetilde{\cX}$, $\widetilde{\cY}$ and $\widetilde{\cZ}$. Then, we have 
\begin{align*}
\divisor_{\widetilde{\cX}^+} (\omega_{\widetilde{\cX}})&  = f_\cX^*(\divisor_{\cX^+}(\omega_\cX)) + \sum a_i E_{\cX,i} \\
\divisor_{\widetilde{\cY}^+} (\omega_{\widetilde{\cY}})&  = f_\cY^*(\divisor_{\cY^+}(\omega_\cY)) +  \sum b_j E_{\cY,j}\\
\divisor_{\widetilde{\cZ}^+} (\omega_{\widetilde{\cZ}})&  = f_\cZ^*(\divisor_{\cZ^+}(\omega_\cZ)) +  \sum c_h E_{\cZ,h}.
\end{align*} As $\divisor_{\cX^+}(\omega_\cX)$, $\divisor_{\cY^+}(\omega_{\cY})$ and $\divisor_{\cZ^+}(\omega_{\cZ})$ do not contain $\Gamma_\cX$, $\Gamma_\cY$ and $\Gamma_{\cZ}$ respectively, we have
\begin{align*}
v_{\Gamma_{\widetilde{\cX}}}(\divisor_{\widetilde{\cX}^+} (\omega_{\widetilde{\cX}}))&  = \sum a_i v_{\Gamma_{\widetilde{\cX}}}(E_{\cX,i}) \geqslant 0 \\
v_{\Gamma_{\widetilde{\cY}}}(\divisor_{\widetilde{\cY}^+} (\omega_{\widetilde{\cY}}))&  = \sum b_j v_{\Gamma_{\widetilde{\cY}}}(E_{\cY,j}) \geqslant 0\\
v_{\Gamma}(\divisor_{\widetilde{\cZ}^+} (\omega_{\widetilde{\cZ}}))&  = \sum c_h v_{\Gamma}( E_{\cZ,h}).
\end{align*}
From the formula \ref{equ valuations product},  $v_{\Gamma}(\divisor_{\widetilde{\cZ}^+} (\omega_{\widetilde{\cZ}}))= v_{\Gamma_{\widetilde{\cY}}}(\divisor_{\widetilde{\cY}^+} (\omega_{\widetilde{\cY}}))+ v_{\Gamma_{\widetilde{\cX}}}(\divisor_{\widetilde{\cX}^+} (\omega_{\widetilde{\cX}}))$, hence we obtain that the log discrepancy of $\Gamma$ with respect to the pair $(\cZ,\Delta_{\cZ})$ is non-negative. Moreover, it is zero if and only if the log discrepancies of $\Gamma_{\widetilde{\cX}}$ and $\Gamma_{\widetilde{\cY}}$ are both zero, namely if and only if $\Gamma_\cX$ and $\Gamma_{\cY}$ are log canonical centers of $(\cX,\Delta_{\cX})$ and $(\cY,\Delta_{\cY})$ respectively. Since for dlt pairs the log canonical centers are the strata of the coefficient $1$ part of the boundary, it follows that any log canonical center of $(\cZ,\Delta_{\cZ})$ is a product of such strata, hence a stratum of the coefficient part $1$ of $(\cZ,\Delta_{\cZ})$.
\end{proof}

\spa{Let $(X,\Delta_X)$ and $(Y,\Delta_Y)$ be dlt pairs over the germ of a punctured curve $C$. We denote by $(Z,\Delta_Z)$ their product, where $\Delta_Z = X \times_C \Delta_Y + \Delta_X \times_C Y$. Let $(\cX,\Delta_{\cX})$ and $(\cY,\Delta_{\cY})$ be semistable good projective dlt minimal models over the pointed curve. }

\begin{theorem}\label{thm product of dual complex of good min dlt models} The product $(Z, \Delta_Z)$ has a semistable good projective dlt minimal model $(\cZ', \Delta_{\cZ'})$ and $\D^{=1}_0(\Delta_{\cZ'}) \simeq \D^{=1}_0(\Delta_\cX) \times \D^{=1}_0(\Delta_\cY)$.
\end{theorem}
\begin{proof}
Let $(\cZ,\Delta_{\cZ})$ be the product of $(\cX,\Delta_{\cX})$ and $(\cY,\Delta_{\cY})$ as in Proposition \ref{prop product of dlt minimal pairs}. Thus $K_{\cZ}+\Delta_{\cZ}$ is semiample and log canonical, so any minimal dlt model over $(\cZ, \Delta_{\cZ})$ will also have semiample log canonical divisor.

Let $\vartheta: \cW \rightarrow \cZ$ be a log resolution given by iterated blow-ups at centers of codimension at least $2$. Then there exists an effective divisor $D$ supported on all of the exceptional divisors, such that $-D$ is $\vartheta$-ample.
\\

\emph{Claim:} there exist $B_\cX$ and $B_\cY$ effective divisors on $\cX$ and $\cY$ whose respective supports contain no log canonical centers of $\Delta_{\cX}$ and $\Delta_{\cY}$, and such that $B_\cX - \epsilon \Delta_{\cX}^{=1}$ and $B_\cY - \epsilon \Delta_{\cY}^{=1}$ are ample, for $\epsilon$ small and rational.
\\

Choose $\epsilon$ small and rational. Then $\Gamma_\cX=\Delta_\cX +B_\cX -\epsilon \Delta_{\cX}^{=1}$ and $\Gamma_\cY=\Delta_\cY +B_\cY -\epsilon \Delta_{\cY}^{=1}$ are effective, and $(\cX,\Gamma_\cX)$ and $(\cY,\Gamma_\cY)$ are klt. Let $\Gamma_{\cW}'$ be the log pullback to $\cW$ of the product $\Gamma_\cZ$ of $\Gamma_\cX$ and $\Gamma_\cY$, namely $K_\cW + \Gamma_{\cW}' = \vartheta^*(K_\cZ + \Gamma_{\cZ})$, we denote by $\Gamma_\cW$ the positive part of $\Gamma_{\cW}'$. Since $(\cZ,\Gamma_\cZ)$ was klt, so is $(\cW,\Gamma_\cW)$. For sufficiently small $\delta$, $(\cW, \Gamma_\cW+\delta D)$ is still klt.

Let $\alpha_i$ be arbitrary small rational coefficients, one for each divisor $\Delta_{\cZ,i}^{=1}$ of $\Delta_\cZ$ with coefficient $1$ and in the special fiber. We will recover the requested dlt model by running an MMP with scaling on the pair $(\cW, \Gamma_\cW+\delta D - \sum \alpha_i \Delta_{\cZ,i}^{=1})$, scaling with respect to an ample divisor $A$ equivalent to $-D$. By \cite{BirkarCasciniHaconEtAl} this MMP terminates in a log terminal model $\phi \colon \cW \dashrightarrow \cZ'$. Moreover, as long as the $\alpha_i$ are small relative to $\delta$, when the MMP terminates it must be the case that every exceptional divisor is contracted.  As running MMP induces birational contractions, the morphism $\psi \colon \cZ' \to \cZ$ is small. Hence $(\cZ', \psi^*\Delta_\cZ)$ is log canonical, and every log canonical center dominates a stratum of the coefficient $1$ part of $(\cZ,\Delta_\cZ)$.

By construction of $\cZ'$, the divisor $-\sum \alpha_i \Delta_{\cZ,i}^{=1} + \mu \phi_* A$ is $\psi$-ample, where $\mu$ is arbitrarily small. But for $\mu$ small enough, we can absorb $\mu \phi_* A$ into the term $\delta D$. Thus in fact $-\sum \alpha_i \Delta_{\cZ,i}^{=1}$ is $\psi$-ample. As a result we can represent $$\cZ' = \text{Proj}_\cZ \bigoplus_{m \geqslant 0} \mathcal{O}\big(-m(\sum \alpha_i \Delta_{\cZ,i}^{=1})\big).$$ At this point we may take an arbitrarily large Veronese subring and assume the $\alpha_i$ are all integers.

Now, we show that $(\cZ', \psi^*\Delta_\cZ)$ is a dlt pair 
by looking at a formal toric model. Indeed, $(\cZ,\Delta_{\cZ})$ is formally locally toric at the log canonical centers, and the condition of being dlt is a formally local property. Moreover, we reduce to check this property at the log canonical centers of  $(\cZ', \psi^*\Delta_\cZ)$ that lie in the special fiber, as the generic fibers of $\cZ'$ and $\cZ$ are isomorphic and the latter is dlt.

Let $z$ be the generic point of the image in $\cZ$ of a log canonical center of $(\cZ', \psi^*\Delta_\cZ)$, hence $z$ is the generic point of a log canonical center of $(\cZ,\Delta_{\cZ})$ and by Proposition \ref{prop product of dlt minimal pairs} it is a stratum of the coefficient $1$ part of $\Delta_{\cZ}$. Let $x$ and $y$ be the generic points of the corresponding strata of $(\cX,\Delta_\cX)$ and $(\cY, \Delta_\cY)$. Let $E_x$ and $E_y$ be the monoids of effective Cartier divisors supported on the strata near $x$ and $y$ respectively. Then the corresponding monoid for $z$ is $\langle E_x \oplus E_y \rangle /(t_x=t_y)$, where $t_x$ and $t_y$ are the respective sums of local equations of strata in the special fibers. This monoid has the form $M \oplus \mathbb{N}^l$, where $l$ is the number of horizontal divisors containing $z$, and $M$ is a monoid of the type considered in Lemma \ref{smalltoric}. The toric variety $T_{M,k}=\Spec k[M \oplus \mathbb{N}^l]$ is a formal local model for $\cZ$ near $z$, so it suffices to consider $$T'=\text{Proj}_{T_{M,k}} \bigoplus_{m \geqslant 0} \mathcal{O}\big(-m(\sum \alpha_i \Delta_{T,i})\big),$$ where the divisors $\Delta_{T,i}$ range over the torus invariant Weil divisors corresponding to the $\Delta_{\cZ,i}^{=1}$.

The divisors corresponding to the $l$ generators of $\mathbb{N}^l$ are Cartier, so their contribution to $T$ is trivial and we can reduce to the case $l=0$. For sufficiently general choices for the $\alpha_i$, the toric variety $T'$ is simplicial, its fan being induced by the simplicial subdivision of the fan of $T_{M,k}$ such that the $\alpha_i$ induce a strictly convex piecewise linear function. Thus, $T'$ is a small $\Q$-factorialization of $T_{M,k}$. By Lemma \ref{smalltoric}, for any such model, $T'$ is a smooth toric variety, hence its invariant divisors are snc. Thus $(\cZ', \psi^*\Delta_\cZ)$ is dlt.

Finally, we compute the dual complex of the coefficient 1 part of $(\cZ', \psi^*\Delta_\cZ)$ by looking at formal toric models again. Locally at the generic points of the log canonical centers, $(\cX,\Delta_{\cX})$ and $(\cY,\Delta_{\cY})$ are snc, hence log-regular, and are also semistable. By Proposition \ref{prop semistability and iso Kato fans} there is a bijective correspondence between pairs of points in the Kato fans of $\cX$ and $\cY$ locally around the log canonical centers and the points in the Kato fan of their product. This induces, by Proposition \ref{prop semistability and skeletons}, a PL homeomorphism between the product of skeletons around the log canonical centres and the skeleton of their product, namely a bijective correspondence between pairs of points in the dual complexes $\Delta_{\cX}^{=1}$ and $\Delta_{\cY}^{=1}$, and points in the skeleton of $(\cZ,\Delta_{\cZ})$ whose projections map to $\Delta_{\cX}^{=1}$ and $\Delta_{\cY}^{=1}$. A face $\sigma_z$ of the skeleton of $\cZ$, where $z$ is a  log canonical centre corresponding to $(x,y)$, corresponds to a prime ideal of the monoid $E_z = \langle E_x \oplus E_y \rangle /(t_x=t_y)$. Thus, we may reduce to consider such faces in the formal toric model $T_{M,k}$, and to study the dual complex of the coefficient 1 part of $(\cZ',\psi^*\Delta_\cZ )$ in the formal model $T'$. But $T'$ is obtained by a simplicial subdivision of $T$. Thus, we conclude that the dual complex $\D^{=1}_0(\psi^*\Delta_\cZ)$ is identified with the product of the dual complexes of the coefficient 1 part of $(\cX,\Delta_{\cX})$ and $(\cY,\Delta_{\cY})$.
\end{proof}

\begin{proof}[Proof of Claim]
Let $(\cX,\Delta_{\cX})$ be a good projective minimal dlt model of $(X,\Delta_X)$ and let $A_\cX$ be an ample divisor on $\cX$. Let $\mathcal{J}$ be the ideal sheaf of $\Delta_{\cX}^{=1}$ and let $$\chi: \cX' = \text{Bl}_\mathcal{J}\cX \rightarrow \cX$$ be the blow-up of $\cX$ with respect to $\mathcal{J}$. Then the trasform $- ((\Delta_{\cX}^{=1})'+ E$ is a $\chi$-ample Cartier divisor, where $(\Delta_{\cX}^{=1})'$ denotes the strict trasform of $\Delta_{\cX}^{=1}$ and $E$ an effective divisor supported on the exceptional divisor of $\chi$. For $\varepsilon$ small positive rational $\chi^*(A_\cX) - \varepsilon (\Delta_{\cX}^{=1} + E) $ is ample. We choose such an $\varepsilon$. For a sufficiently large integer $n$, we can find $$G_{\cX'} \sim n \chi^*(A_\cX) - n \varepsilon (\Delta_{\cX'}^{=1} + E)$$ such that $G_{\cX'}$ is effective and contains no log canonical centers. Then the push-forward $G_\cX \sim n A_\cX - n \varepsilon \Delta_{\cX}^{=1}$ of $G_{\cX'}$ is effective and contains no log canonical centers of $(\cX,\Delta_{\cX})$. In particular $G_\cX$ is Cartier at the log canonical centers of $(\cX,\Delta_{\cX})$.

We can run the same construction for the ideal sheaf of the divisor $G_\cX$. Notice that the blow-up with respect to $G_\cX$ is an isomorphism at the log canonical centers of $(\cX,\Delta_{\cX})$ as there $G_\cX$ is Cartier. We obtain that, for sufficiently small positive rationals $\delta$ and then sufficiently large integers $m$, we can find $H_\cX \sim m A_\cX - m \delta G_\cX$ such that it is effective and not containing any log canonical center of $(\cX,\Delta_{\cX})$. Then $$\frac{1}{\delta}A_\cX - \frac{1}{m \delta}H_\cX \sim G_\cX \sim n A_\cX - n \varepsilon \Delta_{\cX}^{=1} $$ implies that  $$ \Big(\frac{1}{\delta} - n\Big) A_\cX  \sim \frac{1}{m \delta} H_\cX - n \varepsilon  \Delta_{\cX}^{=1}.$$ For $\delta$ sufficiently small the term $\big(\frac{1}{\delta} - n\big)$ is positive, and then for $m$ sufficiently large $B_\cX = \frac{1}{m \delta} H_\cX$ concludes the proof.
\end{proof}

\section{Applications}\label{sec hkahler}
\subsection{Weight functions and skeletons for finite quotients}

\spa{Let $X$ be a connected, smooth  and proper $K$-variety and let $G$ be a finite group acting on $X$. Let $X^{\an}$ be the analytification of $X$. We recall that any point of $X^{\an}$ is a pair $x=(\xi_x,|\cdot|_x)$ with $\xi_x \in X$ and $|\cdot|_x$ an absolute value on the residue field $\kappa(\xi_x)$ that extends the absolute value on $K$. For any point $\xi_x$ of $X$, an element $g$ of the group $G$ induces an isomorphism between the residue fields $\kappa(\xi_x)$ and $\kappa(g \ldotp \xi_x)$, that we still denote by $g$. Then, the action of $G$ extends to $X^\an$ in the following way $$g \ldotp (\xi_x,|\cdot|_x) = ( g \ldotp \xi_x,  |\cdot|_x \circ g^{-1}).$$ In particular the action preserves the sets of divisorial and birational points of $X$.

Let $f: X \rightarrow Y=X/G$ be the quotient map of $K$-schemes, let $f^\an: X^\an \rightarrow Y^\an$ be the map of Berkovich spaces induced by functoriality and let $\tilde{f}: X^\an \rightarrow X^\an/G$ be the quotient map of topological spaces.
\begin{prop} (\cite{Berkovich}, Corollary 5) \label{prop Berk identification of quotient}
Under the above notations, there is a canonical homeomorphism between $X^\an/G $ and $Y^\an$ such that $\tilde{f}$ and $f^\an$ are identified.
\end{prop}
}

\begin{lemma} \label{lemma G invariant forms and weight function}
Let $\omega$ be a $m$-pluricanonical rational form on $X$. If $\omega$ is $G$-invariant, then the weight function associated to $\omega$ on the set of birational points is stable under the action of $G$. 
\end{lemma}
\begin{proof}
Let $x$ be a birational point of $X$ and $g$ an element of $G$. There exist snc models $\cX$ and $\cX'$ over $R$ such that $x \in \Sk(\cX)$ and $g \ldotp x \in \Sk(\cX')$. By replacing them by an snc model $\cY$ that dominates both $\cX$ and $\cX'$, we can assume that both points lies in $\Sk(\cY)$. The weights of $\omega$ at $x$ and $g \ldotp x$ are such that \begin{align*} \weight_{\omega}(g \ldotp x)& = v_{g \ldotp x}(\divisor_{\cY^+}(\omega)) + m  = v_{x}((g^{-1})^*\divisor_{\cY^+}(\omega)) + m  \\ & = v_{x}(\divisor_{\cY^+}(\omega)) + m = \weight_{\omega}(x) \end{align*} as $\omega$ is a $G$-invariant form. Thus we see that birational points in the same $G$-orbit have the same weight with respect to $\omega$. 
\end{proof}

\begin{cor} \label{cor:G inv forms and KS}
Let $\omega$ be a $G$-invariant pluricanonical rational form on $X$. Then the Kontsevich-Soibelman skeleton $\Sk(X,\omega)$ is stable under the action of $G$.
\end{cor}
\begin{proof}
This follows immediately from Lemma \ref{lemma G invariant forms and weight function}.
\end{proof}

\spa{\label{par divisorial repre quotient} Let $y$ be a divisorial point of $Y^\an$ and consider a regular snc $R$-model $\cY$ of $Y$ adapted to $y$, i.e. such that $y$ is the divisorial point associated to $(\cY, E)$ for some irreducible component $E$ of $\cY_k$. We denote by $\cX$ the normalization of $\cY$ inside $K(X) $, where $K(\cY)=K(Y)=K(X)^{G} \hookrightarrow K(X)$. As $X$ is normal and the quotient map $f:X \rightarrow Y$ is finite, we obtain that $\cX$ is an $R$-model of $X$. Moreover, by normality $\cX$ is regular at the generic points of the special fiber $\cX_k$.}

\spa{We assume $\text{char}(k)=0$. We denote respectively by $R$ and by $B$ the codimension $1$ components of the ramification locus and of the branch locus of $f:X \rightarrow Y$. We set $X^+=(X,R_\redu)$ and $Y^+=(Y,B_\redu)$. Then $f$ extends to a morphism of log schemes $f^+:X^+ \rightarrow Y^+$.
}

\begin{lemma} \label{lemma quotient map is log etale}
The logarithmic canonical bundles $\omega_{X^+/K}^{\log}$ and $\omega_{Y^+/K}^{\log}$ are identified via the pullback $(f^+)^*$. 
\end{lemma}
\begin{proof}
From a generalization to higher dimension of \cite{Hartshorne1977}, Proposition IV.2.3, we have that $\omega_{X/K}=f^*(\omega_{Y/K}) \otimes\mathcal{O}_{X}(R)$. It follows that 
$$\omega_{X^+/K}^{\log} = \omega_{X/K} \otimes\mathcal{O}_{X}(R_\redu) = f^*(\omega_{Y/K}) \otimes \mathcal{O}_{Y}(R + R_\redu).$$
In order to study the divisor $R + R_\redu$, we consider one irreducible component of $B$. Let $D_B$ be an irreducible component of $B$, denote by $e$ the ramification index of $f$ at $D_B$, and by $D_R$ the support of the preimage of $D_B$. As $\text{mult}_{D_R}(R)=e-1$, $\text{mult}_{D_R}(f^*(D_B))= e$ and $\text{mult}_{D_R}(R_\redu)=1$, we conclude that $R + R_\redu= f^*(B)$ and finally $\omega_{X^+/K}^{\log} = (f^+)^*(\omega_{Y^+/K}^{\log})$.
\end{proof}
\begin{prop} \label{prop:G inv form and corresponding weights}
Let $\omega$ be a $G$-invariant $R_\redu$-logarithmic $m$-pluricanonical form on $X^+$ and let $\overline{\omega}$ be the corresponding $B_\redu$-logarithmic form on $Y^+$ via pullback. Let $y$ be a divisorial point of $Y^\an$. Then, for any divisorial point $x \in (f^\an)^{-1}(y)$, the weights of $\omega$ at $x$ and of $\overline{\omega}$ at $y$ coincide.
\end{prop}
\begin{proof}
Let $\cY$ be a model of $Y$ over $R$ such that $y$ has divisorial representation $(\cY,E)$ and the model is regular at the generic point of $E$. Let $\cX$ be the normalization of $\cY$ in $K(X)$: as we observed in Section \ref{par divisorial repre quotient}, it is a model of $X$, regular at generic points of the special fiber $\cX_k$. We denote as well by $f:\cX \rightarrow \cY$ the extension of $f$. The preimage of $E$ coincides with the pullback of the Cartier divisor $E$ on $\cX$, hence $f^{-1}(E)$ still defines a codimension one subset on $\cX$. We denote by $F_i$ the irreducible components of $f^{-1}(E)$ and we associate to $F_i$'s their corresponding divisorial valuations $x_i=(\cX,F_i)$. By Lemma \ref{lemma G invariant forms and weight function}, it is enough to prove the result for one of the $x_i$'s. We denote it by $x = (\cX,F)$ and we compare the weights at $y$ and $x$.


We recall that for log-\'{e}tale morphisms the sheaves of logarithmic differentials are stable under pullback (\cite{Kato1994a}, Proposition 3.12). Furthermore, it suffices to check that, locally around the generic point of $F,$ the morphism $\cX^+ \rightarrow \cY^+$ is a log-\'{e}tale morphism of divisorial log structures, to conclude that the weights coincide. To this purpose, we will apply Kato's criterion for log-\'{e}taleness (\cite{Kato1989}, Theorem 3.5) to log schemes with respect to the \'{e}tale topology.

We denote by $\xi_{F}$ the generic point of $F$ and by $\xi_E$ the generic point of $E$. The divisorial log structures on $\cX^+$ and $\cY^+$ have charts $\N$ at $\xi_{F}$ and $\xi_E$. In the \'{e}tale topology, the normalization morphism $\cX^+ \rightarrow \cY^+$ admits a chart induced by $u:\N \rightarrow \N$ where $1 \mapsto m$ for some positive integer $m$:
\begin{center}
$\xymatrix{
\Spec \caO_{\cX, \xi_{F}} \ar[d] \ar[r] & \Spec \Z[\N] \ar[d] \\
\Spec \caO_{\cY, \xi_{E}}  \ar[r] & \Spec \Z[\N] 
}$
\end{center} Firstly, by the universal property of the fiber product, we have a morphism $$\Spec \caO_{\cX, \xi_{F}} \rightarrow \Spec \caO_{\cY, \xi_{E}} \times_{\Spec \Z[\N] } \Spec \Z[\N]$$ and it corresponds to $$ \caO_{\cY, \xi_{E}}  \otimes_{\Z[\N]} \Z[\N] \rightarrow\caO_{\cX, \xi_{F}}.$$ This is a morphism of finite type with finite fibers between regular rings and by \cite{Liu2002}, Lemma 4.3.20 and \cite{Nowak} it is flat and unramified, hence \'{e}tale. One of the two conditions in Kato's criterion for log-\'{e}taleness is then fulfilled. Secondly, the chart $u:\N \mapsto \N$ induces a group homomorphism $u^{\text{gp}}: \Z \mapsto \Z$; in particular, it is injective and it has finite cokernel. Then $u$ satisfies the second condition of Kato's criterion for log-\'{e}taleness. Therefore we conclude that $\weight_{\overline{\omega}}(y)=\weight_\omega(x)$.
\end{proof}

\begin{prop} \label{prop:G inv and homeo KS of quotient}
Let $\omega$ be a $G$-invariant $R_\redu$-logarithmic pluricanonical form on $X$ and let $\overline{\omega}$ be the corresponding $B_\redu$-logarithmic form. Then the canonical homeomorphism between $X^\an/G$ and $Y^\an$ of Proposition \ref{prop Berk identification of quotient} induces the homeomorphism $$\Sk(X,\omega)/G \simeq \Sk(X/G,\overline{\omega}).$$
\end{prop}
\begin{proof}
This follows immediately from Corollary \ref{cor:G inv forms and KS} and Proposition \ref{prop:G inv form and corresponding weights}.
\end{proof}

\spa{Let $X$ be a smooth $K$-variety and let $\omega_X$ be a pluricanonical form on $X$. Let $\pr_j:X^n \rightarrow X$ be the $j$-th canonical projection. Then $$\omega= \bigwedge_{1\leqslant j \leqslant n} \pr_j^*\omega_X$$ is a pluricanonical form on $X^n$ and moreover it is invariant under the action of $\mathfrak{S}_n$. We denote by $\overline{\omega}$ the corresponding form on the quotient $X^n/\mathfrak{S}_n$ as in Lemma \ref{lemma quotient map is log etale}.
\begin{prop}  \label{prop semistability and KS symm quotient}
Assume that the residue field $k$ is algebraically closed.  If $X$ admits a semistable log-regular model or a semistable good dlt minimal model, then the Kontsevich-Soibelman skeleton of the $n$-th symmetric product of $X$ associated to $\overline{\omega}$ is PL homeomorphic to the $n$-th symmetric product of the Kontsevich-Soibelman skeleton of $X$ associated to $\omega_X$.
\end{prop}
\begin{proof}
Iterating the result of Theorem \ref{thm essential skeleton of product via log-reg} and Theorem \ref{thm product of dual complex of good min dlt models}, we have that the projection map defines a PL homeomorphism of Kontsevich-Soibelman skeletons $$\Sk(X^n, \omega) \xrightarrow{\sim} \Sk(X,\omega_X) \times  \ldots \times \Sk(X,\omega_X).$$ Thus, applying Proposition \ref{prop:G inv and homeo KS of quotient} with the group $\mathfrak{S}_n$ acting on the product $X^n$, we obtain that $$\Sk(X^n/\mathfrak{S}_n, \overline{\omega}) \simeq \Sk(X^n,\omega)/\mathfrak{S}_n \simeq  \Sk(X,\omega_X)^n/ \mathfrak{S}_n$$
Since the action on the Kontsevich-Soibelman skeleton $\Sk(X^n,\omega)$ is induced from the symmetric action on $X^n$, and the projections $\pr_j:X^n \rightarrow X$ functorially induce the projections $\overline{\pr_j}: \Sk(X,\omega_X)^n \rightarrow \Sk(X,\omega)$, the action of $\mathfrak{S}_n$ on $\Sk(X,\omega)^n$ is exactly by permutations of the components. Thus, $$ \Sk(X,\omega_X)^n/ \mathfrak{S}_n \simeq \text{Sym}^n(\Sk(X,\omega_X)).$$
\end{proof}
}

\subsection{The essential skeleton of Hilbert schemes of a $\text{K3}$ surface} \label{sect essential sk Hilb}

\spa{Let $S$ be an irreducible regular surface. We consider $\Hilb^n(S)$ the Hilbert scheme of $n$ points on $S$: by \cite{Fogarty} it is an irreducible regular variety of dimension $2n$. Moreover, the morphism $$\rho_{HC}: \Hilb^n(S) \rightarrow S^n/\mathfrak{S}_n$$ that sends a zero-dimensional scheme $Z \subseteq S$ to its associated zero-cycle $\text{supp}(Z)$ is a birational morphism, called the Hilbert-Chow morphism.}

\spa{Let $S$ be a $\text{K}3$ surface over $K$, namely $S$ is a complete non-singular variety of dimension two such that $\Omega_{S/K}^2 \simeq \caO_S \text{ and } H^1(S,\caO_S) = 0$. In particular $S$ is a variety with trivial canonical line bundle. 
\begin{cor}  \label{cor essential skeleton hilb}
Assume that the residue field $k$ is algebraically closed. Suppose that $S$ admits a semistable log-regular model or a semistable good dlt minimal model. Then the essential skeleton of the Hilbert scheme of $n$ points on $S$ is PL homeomorphic to the $n$-th symmetric product of the essential skeleton of $S$ $$\Sk(\Hilb^n(S)) \xrightarrow{\sim} \text{Sym}^n(\Sk(S)).$$ 
\end{cor}
\begin{proof}
This follows immediately from Corollary \ref{prop semistability and KS symm quotient} and the birational invariance of the essential skeleton, \cite{MustataNicaise}, Proposition 4.10.1.
\end{proof}
\begin{prop} \label{prop top essential skeleton Hilb}
If the essential skeleton of $S$ is PL homeomorphic to a point, a closed interval or the $2$-dimensional sphere, then the essential skeleton of $\Hilb^n(S)$ is PL homeomorphic to a point, the standard $n$-simplex or $\CP^n$ respectively. 
\end{prop}
\begin{proof} Applying Corollary \ref{cor essential skeleton hilb}, we reduce to the computation of the symmetric product of a point, a closed interval or the sphere $S^2$. Then, the result is trivially true in the first two cases, and follows from \cite{Hatcher}, Section 4K in the third case.
\end{proof}

}

\subsection{The essential skeleton of generalised Kummer varieties} \label{sect essential sk Kummer}
\spa{Let $A$ be an abelian surface over $K$, namely a complete non-singular, connected group variety of dimension two. Since $A$ is a group variety, the canonical line bundle is trivial and the group structure provides a multiplication morphism $m_{n+1}: A \times A \times \ldots \times A \rightarrow A$ that is invariant under the permutation action of $\mathfrak{S}_{n+1}$, hence it induces a morphism $$\Sigma_{n+1}: \Hilb^{n+1}(A) \xrightarrow{\rho_{HC}} \text{Sym}^{n+1}(A) \rightarrow A$$ by composition with the Hilbert-Chow morphism. Then $\text{K}_n(A) = \Sigma_{n+1}^{-1}(1)$ is called the $n$-th generalised Kummer variety and is a hyper-K\"{a}hler manifold of dimension $2n$ (\cite{Beauville1983}).
}
\spa{In \cite{HalvardHalleNicaise2017}, Proposition 4.3.2, Halle and Nicaise, using Temkin's generalization of the weight function (\cite{Temkina}), prove that the essential skeleton of an abelian variety $A$ over $K$ coincides with the construction of a skeleton of $A$ done by Berkovich in \cite{Berkovich1990}, Paragraph 6.5. It follows from this identification and \cite{Berkovich1990}, Theorem 6.5.1 that the essential skeleton of $A$ has a group structure, compatible with the group structure on $A^\an$ under the retraction $\rho_A$ of $A^\an$ onto the essential skeleton, so the following diagram commutes
\begin{center}
$\xymatrix{
(A^\an)^{n+1} \ar[d]_{(\rho_A)^{n+1}} \ar[r]^-{m_{n+1}^\an} & A^\an \ar[d]^{\rho_A} \\
\Sk(A)^n \ar[r]^-{\mu_{n+1}} & \Sk(A)
}$
\end{center} where $\mu$ denotes the multiplication of $\Sk(A)$.
}
\begin{prop} \label{prop essential skeleton Kummer}
Assume that the residue field $k$ is algebraically closed.  Suppose that $A$ admits a semistable log-regular model or a semistable good dlt minimal model. Then the essential skeleton of the $n$-th generalised Kummer variety is PL homeomorphic to the symmetric quotient of the kernel of the morphism $\mu$, namely $$\Sk(\text{K}_n(A)) \simeq \Sk\big(m_{n+1}^{-1}(1)/\mathfrak{S}_{n+1}\big) \simeq \mu_{n+1}^{-1}(1)/\mathfrak{S}_{n+1}.$$ 
\end{prop}
\begin{proof}
The first homeomorphism follows from the birational invariance of the essential skeleton (\cite{MustataNicaise}, Proposition 4.10.1). We denote by $$L=m_{n+1}^{-1}(1) \quad \text{ and }\quad \Lambda=\mu_{n+1}^{-1}(1).$$ For any choice of an $\mathfrak{S}_{n+1}$-invariant generating canonical form on $A$, it follows from Proposition \ref{prop:G inv and homeo KS of quotient} that $\Sk(L/\mathfrak{S}_{n+1}) \simeq \Sk(L)/\mathfrak{S}_{n+1}$. We reduce to study the quotients $\Sk(L)/\mathfrak{S}_{n+1}$ and $\Lambda/\mathfrak{S}_{n+1}$.

Let $\mathfrak{S}_n'$ and $\mathfrak{S}_n''$ be the subgroups of $\mathfrak{S}_{n+1}$ of the permutations that fix $n$ and $n+1$ respectively. Then $\mathfrak{S}_{n+1}$ is generated by the two subgroups, so its action on $\Sk(L)$ and $\Lambda$ is completely determined by the actions of these subgroups. We consider the following isomorphisms 
\begin{align*}
f_n: &L \xrightarrow{\sim} A^n \quad (z_1,\ldots,z_{n+1}) \mapsto (z_1,\ldots,z_{n-1},z_{n+1}) \\
f_{n+1}: &L \xrightarrow{\sim} A^n \quad (z_1,\ldots,z_{n+1}) \mapsto (z_1,\ldots,z_{n-1},z_{n}).
\end{align*} Then $f_n$ is $\mathfrak{S}_n'$-equivariant, $f_{n+1}$ is $\mathfrak{S}_n''$-equivariant and the morphism $\psi$
\begin{center}
\xymatrixrowsep{0.1pc} $\xymatrix{
& L  \ar[ldd]_{f_{n+1}} \ar[ddr]^{f_n} &\\
& & \\
A^n \ar[rr]^{\psi}& & A^n \\
(z_1,\ldots,z_{n-1},z_{n}) \ar@{|->}[rr]& & (z_1,\ldots, \prod_{i=1}^{n} z_i^{-1}).
}$
\end{center} is equivariant with respect to the action of $\mathfrak{S}_n''$ on the source and of $\mathfrak{S}_n'$ on the target. Hence, we obtain a commutative diagram of equivariant isomorphisms. We denote by $\overline{f}_n$, $\overline{f}_{n+1}$ and $\overline{\psi}$ the isomorphisms induced on the essential skeletons. By Theorem \ref{thm essential skeleton of product via log-reg} and Theorem \ref{thm product of dual complex of good min dlt models} we can identify $\Sk(A^n)$ with $\Sk(A)^n$. Thus, we have the commutative diagram
\begin{center}
\xymatrixrowsep{0.1pc} $\xymatrix{
& \Sk(L)  \ar[ldd]_{\overline{f}_{n+1}} \ar[ddr]^{\overline{f}_n} &\\
& & \\
\Sk(A)^n \ar[rr]^{\overline{\psi}}& & \Sk(A)^n \\
(v_1,\ldots,v_{n-1},v_{n}) \ar@{|->}[rr]& & (v_1,\ldots, \prod_{i=1}^{n} v_i^{-1}).
}$
\end{center} Then the action of $\mathfrak{S}_{n+1}$ on $\Sk(L)$ is induced by the isomorphisms $\overline{f}_n$ and $\overline{f}_{n+1}$ from the actions of $\mathfrak{S}_n''$ and $\mathfrak{S}_n'$ on $\Sk(A)^n$ and these actions are compatible as $\overline{\psi}$ is equivariant.

In a similar way, $\Lambda$ is isomorphic to $n$ copies of $\Sk(A)$ and comes equipped with an action of $\mathfrak{S}_{n+1}$. So, we have equivariant projections $g_n$ and $g_{n+1}$ with respect to $\mathfrak{S}_n'$ and $\mathfrak{S}_n''$. The equivariant morphism that completes and makes the diagram commutative is $\overline{\psi}$. Finally, we have the equivariant commutative diagram
\begin{center}
\xymatrixrowsep{1pc}
\xymatrixrowsep{1pc}
$\xymatrix{
& \Sk(L)  \ar[ld]_{\overline{f}_{n+1}} \ar[rd]^{\overline{f}_{n}}  &\\
\mathfrak{S}_n'' \circlearrowleft\Sk(A)^n  \ar[rr]^{\overline{\psi}}& & \Sk(A)^n\circlearrowright \mathfrak{S}_n' \\
& \Lambda\ar[lu]^{g_{n+1}}\ar[ru]_{g_n} &
}$
\end{center} and we conclude that the quotients $\Sk(L)/\mathfrak{S}_{n+1}$ and $\Lambda/\mathfrak{S}_{n+1}$ are homeomorphic.
\end{proof}

\begin{prop} \label{prop top essential skeleton Kummer}
If the essential skeleton of $A$ is PL homeomorphic to a point, the circle $S^1$ or the torus $S^1 \times S^1$, then the essential skeleton of $\text{K}_n(A)$ is PL homeomorphic to a point, the standard $n$-simplex or $\CP^n$ respectively. 
\end{prop}
\begin{proof}
The case of the point is trivial. For the circle $S^1$, it follows directly from \cite{Morton}, Theorem. To prove the result for the torus $S^1 \times S^1$, we apply \cite{Looijenga}, Theorem 3.4: the action of the symmetric group corresponds to the root system of $A_n$, the highest root is the sum of the simple roots, each with coefficient $1$, and so the quotient is the complex projective space of dimension $n$.
\end{proof}

\subsection{Remarks on hyper-K{\"a}hler varieties} 
\spa{The cases we consider in Proposition \ref{prop top essential skeleton Hilb} and Proposition \ref{prop top essential skeleton Kummer} are motivated by the work of Kulikov, Persson and Pinkham. In \cite{Kulikov} and \cite{PerssonPinkham1981}, they consider degenerations  over the unit complex disk, of surfaces such that some power of the canonical bundle is trivial. They prove that, after base change and birational transformations, any such degeneration can be arranged to be semistable with trivial canonical bundle, namely a \emph{Kulikov degeneration}. Then, they classify the possible special fibers of Kulikov degenerations according to the type of the degeneration.

We recall that the monodromy operator $T$ on $\text{H}^2(X_t, \Q)$ of the fibers $X_t$ of a Kulikov degeneration is unipotent, so we denote by $\nu$ the nilpotency index of $\log(T)$, namely the positive integer such that $\log(T)^\nu=0$ and $\log(T)^{(\nu-1)} \neq 0$. The type of the Kulikov degeneration is defined as the nilpotency index $\nu$ and called type I, II or III accordingly to it.

It follows from \cite{Kulikov}, Theorem II, that the dual complex of the special fiber of a Kulikov degeneration of a $\text{K}3$ surface is a point, a closed interval or the sphere $S^2$ according to the respective type. For a degeneration of abelian surfaces, the dual complex of the special fiber is homeomorphic to a point, the circle $S^1$ or the torus $S^1 \times S^1$ according to the three types (see an overview of these results in \cite{FriedmanMorrison}). In all cases, the dimension of the dual complex is equal $\nu-1$, hence determined by the type.
}

\spa{Hilbert schemes of $\text{K}3$ surfaces and generalised Kummer varieties represent two families of examples of hyper-K\"{a}hler varieties. For a semistable degeneration of hyper-K\"{a}hler manifolds over the unit disk, it is possible to define the type as the nilpotency index of the monodromy operator on the second cohomology group. It naturally extends the definition for Kulikov degenerations.

In \cite{KollarLazaSaccaEtAl2017}, Koll\'{a}r, Laza, Sacc\`{a} and Voisin study the essential skeleton of a degeneration of hyper-K\"{a}hler manifolds in terms of the type. More precisely, in Theorem 0.10, given a minimal dlt degeneration of $2n$-dimensional hyper-K\"{a}hler manifolds, firstly they prove that the dual complex of the special fiber has dimension $(\nu-1)n$, where $\nu$ denotes the type of the degeneration. Secondly, they prove that, in the type III case, the dual complex is a simply connected closed pseudo-manifold with the rational homology of $\CP^n$.

From this prospective, Proposition \ref{prop top essential skeleton Hilb} and Proposition \ref{prop top essential skeleton Kummer} confirm and strengthen their result for the specific cases of Hilbert schemes and generalized Kummer varieties. In particular, we turn the rational cohomological description of the essential skeleton (Theorem 0.10(ii)) into a topological characterization.

For Hilbert schemes associated to some type II degenerations of $\text{K}3$ surfaces, a complementary proof of our result is due to Gulbrandsen, Halle, Hulek and Zhang, see \cite{GulbrandsenHalleHulek2016} and \cite{GulbrandsenHalleHulekEtAl}. Their approach is based on the method of \textit{expanded degenerations}, which first appeared in \cite{Li}, and on the construction of suitable GIT quotients, in order to obtain an explicit minimal dlt degeneration for the associated family of Hilbert schemes.
}
\spa{The structure of the essential skeleton of a degeneration of hyper-K\"{a}hler manifolds is relevant in the context of mirror symmetry and in view of the work of Kontsevich and Soibelman (\cite{KontsevichSoibelmana}, \cite{KontsevichSoibelman}). The SYZ fibration (\cite{StromingerYauZaslow}) is a conjectural geometric explanation for the phenomenon of mirror symmetry and, roughly speaking, asserts the existence of a special Lagrangian fibration, such that mirror pairs of manifolds with trivial canonical bundle should admit fiberwise dual special Lagrangian fibrations. Moreover, the expectation is that, for type III degenerations of $2n$-dimensional hyper-K\"{a}hler manifolds, the base of the SYZ fibration is $\CP^n$ (see for instance \cite{Hwang}).

The most relevant fact from our prospective is that Kontsevich and Soibelman predict that the base of the Lagrangian fibration of a type III degeneration is homeomorphic to the essential skeleton. So, the outcomes on the topology of the essential skeleton we obtain in Proposition \ref{prop top essential skeleton Hilb} and Proposition \ref{prop top essential skeleton Kummer} match the predictions of mirror symmetry about the occurrence of $\CP^n$ in the type III case. 
}

\bibliographystyle{amsalpha}      
\bibliography{References}

\end{document}